\documentclass[11pt]{article}
\usepackage{amsmath, amssymb, amsfonts, amsthm, graphicx, relsize, mathtools, mathrsfs, euscript, bbm, bbold, float, hyperref, enumitem, url, breakurl, stmaryrd, indentfirst}
\usepackage[english]{babel}
\usepackage{titlesec}
\titlelabel{\thetitle.\,\,}
\allowdisplaybreaks
\hypersetup{
colorlinks,
linkcolor={black!100!black},
citecolor={black!100!black},
urlcolor={black!100!black}
}

\let\svthefootnote\thefootnote
\newcommand\blankfootnote[1]{%
\let\thefootnote\relax\footnotetext{#1}%
\let\thefootnote\svthefootnote%
}

\makeatletter
\newtheorem*{rep@theorem}{\rep@title}
\newcommand{\newreptheorem}[2]{%
\newenvironment{rep#1}[1]{%
 \def\rep@title{#2 \ref{##1}}%
 \begin{rep@theorem}}%
 {\end{rep@theorem}}}
\makeatother

\theoremstyle{plain}
\newtheorem{theorem}{Theorem}[section]
\newreptheorem{theorem}{Theorem}
\newtheorem{lemma}[theorem]{Lemma}

\newtheorem{conjecture}[theorem]{Conjecture}

\theoremstyle{definition}

\DeclareMathAlphabet{\mathbbmsl}{U}{bbm}{m}{sl}

\DeclareMathAlphabet{\mathpzc}{OT1}{pzc}{m}{it}
\DeclareMathAlphabet{\mathsfit}{T1}{\sfdefault}{\mddefault}{\sldefault}\SetMathAlphabet{\mathsfit}{bold}{T1}{\sfdefault}{\bfdefault}{\sldefault}

\makeatletter
\DeclareRobustCommand\widecheck[1]{{\mathpalette\@widecheck{#1}}}
\def\@widecheck#1#2{%
\setbox\z@\hbox{\m@th$#1#2$}%
\setbox\tw@\hbox{\m@th$#1%
\widehat{%
\vrule\@width\z@\@height\ht\z@
\vrule\@height\z@\@width\wd\z@}$}%
\dp\tw@-\ht\z@
\@tempdima\ht\z@ \advance\@tempdima2\ht\tw@ \divide\@tempdima\thr@@
\setbox\tw@\hbox{%
\raise\@tempdima\hbox{\scalebox{1}[-1]{\lower\@tempdima\box
\tw@}}}%
{\ooalign{\box\tw@ \cr \box\z@}}}
\makeatother

\usepackage[svgnames]{xcolor}
\usepackage{tikz}
\usetikzlibrary{decorations.markings}
\usetikzlibrary{shapes.geometric}

\pgfdeclarelayer{edgelayer}
\pgfdeclarelayer{nodelayer}
\pgfsetlayers{edgelayer,nodelayer,main}

\tikzstyle{none}=[inner sep=0pt]

\tikzstyle{rn}=[circle,fill=Red,draw=Black,line width=0.8 pt]
\tikzstyle{gn}=[circle,fill=Lime,draw=Black,line width=0.8 pt]
\tikzstyle{yn}=[circle,fill=Yellow,draw=Black,line width=0.8 pt]

\tikzstyle{simple}=[-,draw=Black,line width=2.000]
\tikzstyle{arrow}=[-,draw=Black,postaction={decorate},decoration={markings,mark=at position .5 with {\arrow{>}}},line width=2.000]
\tikzstyle{tick}=[-,draw=Black,postaction={decorate},decoration={markings,mark=at position .5 with {\draw (0,-0.1) -- (0,0.1);}},line width=2.000]

\tikzstyle{newstyle}=[-,draw=Green,line width=2.000]

\begin{document}

\title{\bf{\LARGE   On     saturation numbers of complete  multipartite graphs and even cycles}\\[9mm]}

\author{
Ali Mohammadian \qquad   Milad Poursoltani \qquad  Behruz Tayfeh-Rezaie \\[4mm]
School of Mathematics,\\
Institute for Research in Fundamental Sciences (IPM),\\
P.O. Box 19395-5746, Tehran, Iran\\[3mm]
\href{mailto:ali\_m@ipm.ir}{ali\_m@ipm.ir} \qquad      \href{mailto:mpour@ipm.ir}{mpour@ipm.ir} \qquad \href{mailto:tayfeh-r@ipm.ir}{tayfeh-r@ipm.ir}\\[9mm]}

\date{}

\maketitle

\begin{abstract}
Given positive integer $n$ and graph   $F$, the saturation number       $\mathrm{sat}(n, F)$   is the minimum  number of edges in  an  edge-maximal $F$-free graph on $n$ vertices.
In this paper,   we determine asymptotic behavior of    $\mathrm{sat}(n, F)$   when $F$ is either a complete multipartite graph or a cycle graph whose   length is  even and large enough.
This extends  a    result by  Bohman,   Fonoberova, and  Pikhurko from  2010 as well as  partially resolves a conjecture of  F\"uredi and  Kim  from 2013. \\[1mm]


\end{abstract}

\section{Introduction}

Extremal problems are among the most studied topics in graph theory.
They  are very popular and have a huge literature.
Here,    we concentrate on the saturation problem on graphs.
Throughout this paper, all graphs are assumed to be finite, undirected,  and without loops or multiple edges.  Let $F$ be  a given  graph. We say that a graph $ G $ is {\it $ F $-free} if $ G $ has no subgraph isomorphic to $ F $. In   1941, Tur\'an \cite{tur} posed one of the  foundational problems   in extremal graph theory  which was about the  maximum number
of edges in   an $ F $-free   graph on $n$ vertices. Later, in 1949, Zykov \cite{zyk} introduced a dual idea   which asks for  the minimum  number of edges in an  edge-maximal  $F$-free graph  on  $n$ vertices.
Formally,   a  graph $G$ is said to be   {\it  $F$-saturated}
if $G$ is $F$-free and, for any  two nonadjacent vertices $x$ and $y$ of $G$, the  graph resulting   from $G$ by joining $x$ and $y$ with  an edge    is not  $F$-free.
The  minimum number of edges  in  an  $F$-saturated  graph on $n$ vertices  is  denoted by $\mathrm{sat}(n, F)$.

Let $ K_s $ denote  the complete   graph on   $s$ vertices and $ K_{s_1, \ldots, s_r} $ denote   the complete $r$-partite graph with parts of sizes $s_1, \ldots, s_r$.
Erd\H{o}s, Hajnal, and Moon     \cite{erd}   proved  that
\begin{align}\label{kn}\mathrm{sat}(n, K_s)=(s-2)n-{{s-1}\choose{2}},\end{align}  where $n\geqslant s\geqslant 2$.
Generally,  the exact value  of   $ \mathrm{sat}(n, K_{s_1, \ldots, s_r})$   is  still unknown.
K\'aszonyi  and   Tuza    \cite{kas}  proved that
\begin{align*}
\mathrm{sat}(n, K_{1, s})=\left\{\begin{array}{llll}\vspace{-4.5mm}&\\
\mathlarger{{{s}\choose{2}}+{{n-s}\choose{2}}} & \quad  \mbox{{\large if }} \mathlarger{s+1\leqslant n\leqslant\frac{3s}{2}}\mbox{;}\\
\vspace{-1mm}\\
\mathlarger{\left\lceil\frac{s-1}{2}n-\frac{s^2}{8}\right\rceil}   & \quad    \mbox{{\large if }} \mathlarger{n\geqslant\frac{3s}{2}}\mbox{.}\\\vspace{-4.25mm}&
\end{array}\right.
\end{align*}
As an asymptotic result, Bohman, Fonoberova, and Pikhurko \cite{pikh} established  that
\begin{align}\label{34}
\mathrm{sat}(n, K_{s_1, \ldots, s_r})=\left(s_1+\cdots+s_{r-1}+\frac{s_r-3}{2}\right)n+O\left(n^{\frac{3}{4}}\right),
\end{align}
where  $r\geqslant2$ and   $1\leqslant s_1\leqslant\cdots\leqslant s_r$.
In     Section \ref{scmg} of this paper, we extend  \eqref{34} by proving the following  theorem.

\begin{theorem}\label{mainCMG}
For any   fixed  integers   $r\geqslant2$ and    $1\leqslant s_1\leqslant\cdots\leqslant s_r$,
$$\mathrm{sat}(n, K_{s_1, \ldots, s_r})=\left(s_1+\cdots+s_{r-1}+\frac{s_r-3}{2}\right)n+O(1).$$
\end{theorem}

\noindent  Our proof for  Theorem \ref{mainCMG}  is    short and simple.  We   recall  that     Theorem \ref{mainCMG}  has been  established    in \cite{pikh} provided  $s_{r-1}\neq s_r$ by    a rather long and complicated proof.

Let $C_k $ denote  the cycle   graph on   $k$ vertices. Generally, it seems hard  to determine  the   exact value  of  $\mathrm{sat}(n, C_k)$. Until now, only  the exact values of $\mathrm{sat}(n, C_k)$ are known   for  $k\in\{3, 4, 5\}$.
It follows from  \eqref{kn} that $\mathrm{sat}(n, C_3)=n-1$ for every integer $n\geqslant3$. Also,
we know that $\mathrm{sat}(n, C_4)=\lfloor(3n-5)/2\rfloor$ for every integer $n\geqslant5$  which   is due to  Ollmann \cite{oll} and  $\mathrm{sat}(n, C_5)=\lceil10(n-1)/7\rceil$ for every integer $n\geqslant21$  which   is due to     Chen \cite{chen}.
For $C_6$,
Lan,    Shi,    Wang, and   Zhang  \cite{lan} showed  that
$4n/3-2\leqslant \mathrm{sat}(n, C_6)\leqslant (4n+1)/3$ for  every integer  $n\geqslant9$. For longer   cycles, F\"uredi and  Kim  \cite{fk} established  that
\begin{align}\label{fur}\frac{k+3}{k+2}n-1<\mathrm{sat}(n, C_k)< \frac{k-3}{k-4}n +{k-4\choose 2},\end{align}
where  $k\geqslant 7$ and $n\geqslant 2k-5$. They   proposed the following conjecture.

\begin{conjecture}[\cite{fk}]\label{conjfur}
There exists a constant $k_0$ such that $$\mathrm{sat}(n, C_k)= \frac{k-3}{k-4}n +O(k^2)$$ holds  for  all integers $k\geqslant k_0$.
\end{conjecture}

\noindent In    Section \ref{sec} of this paper, we  partially resolve Conjecture    \ref{conjfur} by    proving   the following  theorem.

\begin{theorem}\label{mainEC}
For each fixed  even     integer   $k\geqslant28$,
$$\mathrm{sat}(n, C_k)= \frac{k-3}{k-4}n +O(1).$$
\end{theorem}

We refer to    \cite{fau} for a survey of  results on saturation numbers of   graphs. We just recall the following   conjecture that is made  by  Tuza    \cite{tuz}.

\begin{conjecture}[\cite{tuz}]\label{conj}
For every  graph  $F$,  there  exists   a constant $c_{_{\mathlarger{F}}}$  such
that  $\mathrm{sat}(n, F)=c_{_{\mathlarger{F}}} n+O(1)$.
\end{conjecture}

\noindent Obviously,  Theorems \ref{mainCMG} and \ref{mainEC}     verify  Conjecture \ref{conj} for   complete  multipartite graphs and   cycle graphs  whose   length are   even and large enough.

\section{Preliminaries}\label{prelim}

At first, let   us fix   some   standard notation  and terminology    of  graph  theory.
Let $G$ be a graph. The vertex set and the edge  set of  $G$ are   denoted by $V(G)$ and   $E(G)$, respectively.
For a subset     $W$ of $V(G)$,  write $G[W]$  for  the subgraph of $G$ induced by $W$ and    set       $$N_G(W)=\big\{v\in V(G) \, \big| \,   \text{$v$   is adjacent to all   vertices in $W$}\big\}.$$  For the sake of convenience, we write $N_G(v_1, \ldots, v_k)$ instead of $N_G(\{v_1, \ldots, v_k\})$ and we set    $N_G[v]=N_G(v)\cup\{v\}$.
The   {\it degree}  of a vertex $v$ of $G$  is defined    as $|N_G(v)|$ and is  denoted   by $ \deg_G(v)$.
The {\it distance} between  two vertices $ u$ and $v$ of $ G$,  denoted by $d_G(u, v)$,   is defined as the minimum  number of edges in   a   path  in   $G$   connecting $u$ and $v$.
For two subsets $ S$ and $ T$   of $V(G)$, the distance  between $S$ and  $T$ is   $\min\{ d_G(x, y) \, | \,  x\in S \text{ and }  y\in T\}$.
Also, we denote by $ E_{G}(S, T)$ the set of all edges
with endpoints in both $ S$ and $ T$.
Recall that a  {\it starlike}  tree  is a tree obtained from at least two
paths that are completely vertex-disjoint, except that they
all share exactly one  endpoint in common,  which  is called the {\it center} of the  starlike tree.

We  recall  an   upper bound on $\mathrm{sat}(n, F)$ which was  proved by K\'aszonyi  and   Tuza  in    \cite{kas}.

\begin{theorem}[\cite{kas}]\label{th:known}
Let $ F$ be a graph and $ S $ be an independent set in $ F$ with maximum possible size. Also, let
$ b = |V(F)|-|S|-1$ and
$ d= \min\{|N_F(x)\cap S| \,|\, x\in V(F)\setminus S\}$. Then,
\[
\mathrm{sat}(n, F)\leqslant \frac{2b +d-1}{2}n-\frac{b(b+d)}{2}.
\]
\end{theorem}

\section{Complete  multipartite graphs}\label{scmg}

In this section, we provide   a proof for Theorem  \ref{mainCMG}.
We  first set up  some notation   to be used  throughout the section.
We always assume that $r\geqslant2$ and  $1\leqslant s_1\leqslant\cdots\leqslant s_r$.
Set   $F=K_{s_1,s_2 ,\ldots, s_r}$ and    assume that  $H$ is an $F$-saturated graph on $n$ vertices.
Let   $s= s_1+\cdots+s_r$,
$$A=\Big\{ x\in V(H)   \,  \Big|  \,   \deg_H (x) \leqslant 2s-s_r - 4\Big\},$$  and
$$\mathcal{A}=\Big\{M\subseteq A  \,  \Big|  \, \big|N_H(u, v)\big|\leqslant s-s_r - 2   \text{ for any  two  distinct vertices } u, v\in M\Big\}.$$

\begin{lemma}\label{bigneib}
Let $a, b, c, g$ be positive integers with   $b=(c-1)g^{(c-1)g^c+1}+(c-1)g^c+2$ and $g=4^a$.    Also, let  $B$ be   an independent  set in $H$    such that  $|B|=b$ and $\deg_H(v)\leqslant a$ for every  vertex  $v\in B$.
Then,  there exists a subset   $C\subseteq B$ such that $|C|=c$ and $|N_H(C)|\geqslant s-s_r-1$. Moreover, $|N_H(C)|>|N_H(B)|$   if  $a=s-3$.
\end{lemma}

\begin{proof}
Let  $B_1$ be an arbitrary subset of $B$ with  $|B_1|=(c-1)g^c+1$ and set   $B_2=B\setminus B_1$.
For two distinct vertices    $x, y\in B$,   if we join     $x$ and  $y$ with  an edge and call the resulting graph by $H+xy$, then at least one  copy of $F$  can be found   in $H+xy$.
Fix   such a   copy $F_{xy}$ of $F$ in $H+xy$.  Let  $T_{x}^{y}$ and $T_{y}^{x}$ denote the   subsets of $V(F_{xy})$ such that $T_{x}^{y}\cup\{x\}   $ and $ T_{y}^{x}  \cup\{y\}$ are two  parts of  $F_{xy}$.  Set $S_{xy}=V(F_{xy})\setminus ( T_{x}^{y} \cup T_{y}^{x}  \cup\{x,y\}  )$.

For  a fixed  vertex $x\in B$, the number of choices for pairs  $(T_{y}^{x}, S_{xy})$ is at most $g$, since  $T_{y}^{x}\cup S_{xy}\subseteq N_H(x)$ and $\deg_H (x) \leqslant a$.
Therefore, the number of choices for pairs $(T_{y}^{x},S_{xy})$,  for all vertices  $x, y \in B_1$,  is at most $g^{(c-1)g^c+1}$. As  $|B_2|=(c-1)g^{(c-1)g^c+1}+1$,   by  the pigeonhole principle,    there exists a subset $B_2' \subseteq B_2$ of size $c$  such that
\begin{align}\label{Equ1}
T_{u}^{x}= T_{v}^{x} \quad \text{ and } \quad  S_{xu}=S_{xv}
\end{align}
for  all     vertices   $x\in B_1$ and    $u,v\in B_2'$.
Similarly,
the number of choices for pairs $(T_{x}^{y},S_{xy})$,  for all vertices  $x\in B_1$ and $y\in B_2'$,  is at most $g^{c}$. As
$|B_1|=(c-1)g^c+1$,  by  the pigeonhole principle,   there exists  a subset $B_1' \subseteq B_1$ of  size $c$  such that
\begin{align}\label{Equ2}
T_{u}^{y}= T_{v}^{y} \quad \text{ and } \quad S_{yu}=S_{yv}.
\end{align}
for  all     vertices   $y\in B_2'$ and $u,v\in B_1'$.
From   \eqref{Equ1} and \eqref{Equ2}, we deduce the following  statements:  (i) For any  vertices  $x\in B_1'$ and     $y\in B_2'$, all   subsets $S_{xy}$ are the same, say  $S$; (ii)   For a fixed vertex  $x\in B_1'$,  all subsets
$T_{w}^{x}$ for   vertices  $w\in B_2'$ are the same, say $T^x$;
(iii)   For a fixed vertex  $y\in B_2'$,  all subsets
$T_{w}^{y}$ for   vertices  $w\in B_1'$ are the same, say $T_y$.
Note that,  for  all    vertices  $x\in B_1'$ and $y\in B_2'$,  we have  $T^x\subseteq N_H(x)$ and    $T_y\subseteq N_H(y)$. Moreover, for  all      vertices  $x\in B_1'$ and $y\in B_2'$,
every vertex in  $T^x$ is adjacent to  every  vertex in   $T_y\cup S$  and
every vertex in  $T_y$ is adjacent to  every  vertex in   $T^x\cup S$.

For two fixed    vertices  $x_0\in B_1'$ and $y_0\in B_2'$,  there are    two distinct  indices    $ i,  j\in\{1, \ldots, r\}$  such  that
$|  T^{x_0}_{y_0} \cup \{y_0\}  |=s_i$ and
$|  T_{x_0}^{y_0} \cup \{x_0\}|=s_j$.
According to the discussion mentioned above,    $|T^{x}|=s_i-1$ and $|T_{y} |=s_j-1$ for all    vertices  $x\in B_1'$ and $y\in B_2'$, and moreover,  $|S|=s-s_i-s_j$.
Let $\widehat{T}=\bigcup_{x\in B_1'} T^x$ and $\widecheck{T}=\bigcup_{y\in B_2'} T_y$.
If $|\widehat{T}|\geqslant s_i$ and $|\widecheck{T}|\geqslant s_j$, then
one may find  a copy of $F$ in $H$ having a   part of size $s_i$   in  $\widehat{T}$,   a    part of size $s_j$ in  $\widecheck{T}$,  and  a    part of size $s_k$ in    $S$ for all indices  $k\in\{1, \ldots, r\}\setminus\{i, j\}$.
However, this  contradicts the $F$-freeness of $H$.  So, without loss of generality, we may assume that $|\widehat{T}|= s_i-1$ which means that $\widehat{T}=T^x$ for every  vertex $x\in B_1'$.
We now have $|N_H(B_1')|\geqslant |\widehat{T}|+|S|= (s_i-1)+(s-s_i-s_j)\geqslant s-s_r-1$. As $B_1'\subseteq B$ and $|B_1'|=c$, we are done.

For the `moreover' statement,  let  $a=s-3$.
If $\widehat{T}\subseteq N_H(B)$, then,  for  every  vertex $y\in B_2'$, $\deg_H(y)\geqslant |\widehat{T}|+|T_y|+|S|=  (s_i-1)+(s_j-1)+(s-s_i-s_j)=a+1$, a contradiction. So,  $\widehat{T}\not\subseteq N_H(B)$ which yields that $|N_H(B_1')|>|N_H(B)|$, as required.
\end{proof}

\begin{lemma}\label{MO1}
For   every  set  $M\in\mathcal{A}$,   $|M|=O(1)$.
\end{lemma}

\begin{proof}
We may assume that $2s-s_r-4\geqslant 0$.
By employing a greedy algorithm,   we get   an independent set $M' \subseteq M$ so  that $|M'|\geqslant |M|/(2s-s_r-3)$. Letting   $g=4^{2s-s_r-4}$, suppose by way of contradiction that
$|M'|>  g^{g^{2}+1}+g^2+2$.
By applying  Lemma \ref{bigneib} for $a=2s-s_r-4$,  an arbitrary subset of $M'$ of size $b=g^{g^{2}+1}+g^2+2$, and $c=2$,   we find  two distinct vertices $u,v\in M'$  so  that $|N_H(u,v)|\geqslant s-s_r-1$ which is a   contradiction to the definition of $M$.
Hence,    \begin{equation*}|M|\leqslant (2s-s_r - 3)|M'|\leqslant (2s-s_r - 3)(g^{g^{2}+1}+g^2+2).\qedhere\end{equation*}
\end{proof}

The following Lemma was proved in \cite[Lemma 7]{pikh} using a Ramsey argument.  For the sake of self-contentedness, we present here a proof using Lemma \ref{bigneib}.

\begin{lemma}\label{CO1}
Let  $D=\{v\in V(H) \, | \,    \deg_H(v)\leqslant s-3\}$. Then,  $|D|=O(1)$.
\end{lemma}

\begin{proof}
We may assume that $s\geqslant 3$. Using a greedy algorithm,  we find an independent set $D' \subseteq D$ such that $|D'|\geqslant |D|/(s-2)$. Letting   $g=4^{s-3}$,
define the function  $\phi(k)=(k-1)g^{(k-1)g^k+1}+(k-1)g^k+2$. Also, define   $\phi^{(0)}(k)=k$    and $\phi^{(i)}(k)=\phi(\phi^{(i-1)}(k))$ for all  integers $i\geqslant1$.

Toward  a  contradiction, suppose  that   $|D'|>\phi^{(s-2)}(1)$. Consider an arbitrary subset $D_0\subseteq D'$ with $|D_0|=\phi^{(s-2)}(1)$.  By applying $s-2$ times     Lemma \ref{bigneib}  gives us    subsets  $D_0\supseteq D_1\supseteq  \cdots \supseteq D_{s-2}$
such that $|D_i|=\phi^{(s-i-2)}(1)$ and  $|N_H(D_{i})|>|N_H(D_{i-1})|$ for $i=1, \ldots, s-2$.
We have  $|D_{s-2}|=1$  and $|N_H(D_{s-2})|\geqslant s-2$. This means that    $D$ contains a vertex  of
degree at least $s-2$,  a contradiction.
Thus,  $|D|\leqslant (s-2)|D'|\leqslant (s-2)\phi^{(s-2)}(1)$.
\end{proof}

We are now in a position to prove our  main result of this section. Recall Theorem \ref{mainCMG}.

\begin{reptheorem}{mainCMG}
For any   integers   $r\geqslant2$ and    $1\leqslant s_1\leqslant\cdots\leqslant s_r$,
$$\mathrm{sat}(n, K_{s_1, \ldots, s_r})=\left(s_1+\cdots+s_{r-1}+\frac{s_r-3}{2}\right)n+O(1).$$		
\end{reptheorem}

\begin{proof}
By   Theorem \ref{th:known}, $$\mathrm{sat}(n, F)\leqslant \frac{2s-s_r-3}{2}n-\frac{(s-1)(s-s_r-1)}{2}.$$ Hence, it remains to prove  that
$$\mathrm{sat}(n, F)\geqslant \frac{2s-s_r-3}{2}n+O(1).$$
For this,  we should   show that $|E(H)|\geqslant \frac{2s-s_r-3}{2}n+O(1)$.
If $A=\varnothing$, then  $|E(H)|\geqslant \frac{2s-s_r-3}{2}n$ and we  are done. So, we assume that $A\neq\varnothing$.
Let  $M$ be a maximal element of $\mathcal{A}$  and set $N=\bigcup_{v\in M}N_H(v)$. Also, let
\begin{align*}
X&= \big\{ v\in V(H)\setminus N \, \big| \,   \deg_H(v) \geqslant 2s-s_r - 3\big\},\\
Y&= \big\{ v\in V(H)\setminus N \, \big| \,   s-2 \leqslant\deg_H(v) \leqslant 2s-s_r - 4\big\},\\
Z&= \big\{ v\in V(H)\setminus N \, \big| \,   \deg_H(v) \leqslant s-3\big\}.
\end{align*}
By the  maximality of $M$, each vertex in $Y$ is adjacent to   at least $s- s_r - 1$ vertices  in $N$.
Also, it follows from  Lemmas   \ref{MO1}  and  \ref{CO1}  that $|N|=O(1)$ and  $|Z|=O(1)$, implying that   $|X|+|Y|=n-|M\cup N|-|Z|=n+O(1)$.
So, we may write
\begin{align*}
|E(H)|
&\geqslant  E_H\big(X, V(H)\big)   + E_H\big(Y, V(H)\big) \\
&\geqslant  \sum_{v \in X} \frac{\deg_H(v)}{2} + \left((s- s_r - 1)|Y|+ \sum_{v \in Y} \frac{\deg_H(v)-(s-s_r -1)}{2}\right)\\
&\geqslant    \frac{2s-s_r -3}{2}|X| +   (s- s_r - 1)|Y|+ \frac{s_r -1}{2}|Y| \\
&=\frac{2s-s_r -3}{2}\big(|X|+|Y|\big) \\
&=\frac{2s-s_r -3}{2}n+O(1).
\qedhere\end{align*}
\end{proof}

\section{Even cycles}\label{sec}

In this section, we prove  Theorem  \ref{mainEC}.   For   clarity, we divide this section into two subsections.

\subsection{General properties}

In this subsection, we state and prove  auxiliary   lemmas    about the structures  of   $C_{k}$-saturated graphs.  The first one  is already known in the literature. To be self-contained, we include its proof here.

Before we proceed,  note that  a graph  is $C_{k}$-saturated if and only if  each    pair of nonadjacent vertices of the  graph   are connected by a  path of    length  $k-1$.
In particular, every   $C_{k}$-saturated  graph is a    connected graph with   diameter at most $k-1$.

\begin{lemma}{\rm (\cite{bcepst})}\label{neigborone}
Let  $k\geqslant5$ and  let  $G$ be a $C_{k}$-saturated graph with  at least $3$ vertices. Also, let $x, x'$ be two distinct  degree-one  vertices of $G$ with   $N_G(x)=\{y\}$ and   $N_G(x')=\{y'\}$ for some   vertices  $y, y'\in V(G)$. Then, the following   statements hold.
\begin{itemize}[noitemsep,  topsep=0pt]
\item[\rm{(i)}] If $v\in N_G[y]\setminus\{x\}$, then $\deg_G(v)\geqslant 3$. In particular, $y\neq y'$.
\item[\rm{(ii)}]  If  $y$  is adjacent to $y'$, then $\deg_G(y)\geqslant 4$.
\end{itemize}
\end{lemma}

\begin{proof}
Note that  $G$ is  a connected   graph with  at least $3$ vertices.

\hspace{-3.7mm} (i)   Let $v\in N_G[y]\setminus\{x\}$. If $\deg_G(v)=1$, then $v\neq y$ and    there is only one path between $x$ and $v$ whose length is $2$, a contradiction.  Suppose  toward a contradiction that   $\deg_G(v)= 2$. If $v=y$, then       $N_G(v) = \{x, w\}$ for some vertex $w$,  and     there is only one path between $x$ and $w$ whose length is $2$, a contradiction.  Therefore, $v\neq y $ and   $N_G(v) = \{y, w\}$ for some vertex $w$. Since $x$ and $w$ are connected by a  path of    length  $k-1$,   there is a path $P$  between $y$ and $w$ of   length   $k-2$. As  $k\geqslant5$ and $\deg_G(v)= 2$, $P$ does not pass through $v$. This means that the cycle  $vyPwv$  has   length $k$,   a contradiction.

\hspace{-3.7mm} (ii)   We know from Part   (i) that   $\deg_G(y)\geqslant 3$.  Working  toward  a contradiction, suppose  that  $\deg_G(y)=3$ and   $N_G(y)=\{x, y',w\}$ for some vertex $w$.   Since $x'$ and $w$ are connected by a  path of    length  $k-1$,   there is a path $P$  between $y'$ and $w$ of    length   $k-2$. As  $k\geqslant5$ and $\deg_G(y)= 3$,
$P$ does not pass through $y$. This means that the cycle  $ywPy'y$   has   length $k$,   a contradiction.
\end{proof}

\begin{lemma}\label{closedfener}
Let $k\geqslant5$ and  let $G$ be a $C_{k}$-saturated graph. Also, let $P=u_0u_1\ldots u_\ell$ be a path on degree-two vertices of $G$ with   $\ell\geqslant2$.
Then,  $u_0$ and $u_\ell$ have no common neighbor in  $V(G)\setminus V(P)$.
\end{lemma}

\begin{proof}
Since  $G$ is a connected  graph which    is not a cycle, $u_0u_\ell\not\in E(G)$.
Toward  a contradiction, suppose  that  $u_0$ and $u_\ell$ have a common neighbor outside $P$. There are exactly two  paths between  $u_0$  and  $u_\ell$  whose  lengths are  $2$ and $\ell$. As   $k\geqslant5$ and $G$ is  $C_{k}$-saturated, we derive  that    $\ell=k-1$. So,  there are exactly two  paths between  $u_1$ and  $u_\ell$  whose  lengths are  $3$ and $k-2$, a  contradiction.
\end{proof}

\begin{lemma}\label{pkgeneral}
Let $k\geqslant5$ and  let $G$ be a $C_{k}$-saturated graph.
Also, let  $u_0u_1\ldots u_{r}$ and  $v_0v_1\ldots v_{s}$ be two  vertex-disjoint paths of $G$ on degree-two vertices      and let   $r+s\geqslant k-1$.
If $u_0$ and $v_0$ have a common neighbor, say $w$,  then  there exists  a path between  $w$ and   $u_r$  of length $k-r-3$  which does not pass  through $u_0$.
\end{lemma}

\begin{proof}
Let  $u_0$ and $v_0$ have a common neighbor, say $w$.
Since  $u_0v_1\not\in E(G)$ and $G$ is $C_{k}$-saturated, there should be  a path $P$  between  $u_0$ and   $v_1$  of  length $k-1$.
If  $u_1, v_2\in V(P)$, then   it follows from   $u_rv_s\not\in E(G)$ that  the  length of $P$ is at least
$r+(s-1)+2\geqslant k$, a contradiction.
If  $w, v_2\in V(P)$, then   $P=u_0wQv_sv_{s-1}\ldots v_1$ for some path  $Q$ of length $k-s-1$ which  means that   $wQv_sv_{s-1}\ldots v_0w$   is  a cycle of length $k$,
a contradiction.
Thus, in view of  $k\geqslant5$,  we must have    $u_1, v_0\in V(P)$. Then,  $P=u_0u_1\ldots u_rRwv_0v_1$ for some path  $R$ of length $k-r-3$ which does not pass  through $u_0$. So,   $R$ is  the desired path between  $w$ and   $u_r$.
\end{proof}

\begin{lemma}\label{book}
Let $k\geqslant7$ and  let $G$ be a $C_{k}$-saturated graph.
Also, let  $u_0u_1\ldots u_{k-4}$ and  $v_0v_1\ldots v_{k-4}$ be two  vertex-disjoint paths of $G$ on degree-two vertices.
Then,  $u_0$ and $v_0$ have no common neighbor.
\end{lemma}

\begin{proof}
If  $u_0$ and $v_0$ have a common neighbor, say $w$,   then   Lemma \ref{pkgeneral} guarantees the existence of      a path between  $w$ and  $u_{k-4}$    of length   $k-(k-4)-3=1$  which contradicts Lemma \ref{closedfener}.
\end{proof}

\subsection{Proof of Theorem \ref{mainEC}}

In this subsection, we  provide  a proof for Theorem \ref{mainEC}.  We first set up  some notation   to be used throughout the subsection.
Fix  an   integer  $\ell\geqslant14$  and   put $k=2\ell$,    $c=6\ell-5$,  and $\alpha=1/(2\ell-4)$.
Also,    assume that    $G$ is a $C_{2\ell}$-saturated graph.  In view of   the upper bound given in \eqref{fur},   we should  prove  that $|E(G)|\geqslant (1+\alpha)|V(G)|+O(1)$.

Let the  vertex   $v\in V(G)$ have  the property  that each vertex of $G$ with  distance at most $\ell-2$ from  $v$ has degree less than $c$.
For such a vertex $v$, define  $D(v)=\{x\in V(G) \, | \, d_G(v, x)\leqslant \ell-1\}$.

Let $u_0u_1\ldots u_{2\ell-4}$ be a  path of length $2\ell-4$ on degree-two  vertices of $G$. For
the middle  vertex $u_{\ell-2}$, we have $D(u_{\ell-2})=\{u_0,u_1, \ldots, u_{2\ell-4},v,w\}$, where $v$ is  the   neighbor  of $u_0$  other than $u_1$ and $w$ is  the   neighbor  of $u_{2\ell-4}$  other than $u_{2\ell-5}$.
Further,  for any  two vertex-disjoint   paths  $u_0u_1\ldots u_{2\ell-4}$ and $v_0v_1\ldots v_{2\ell-4}$   of length   $2\ell-4$ on degree-two  vertices of $G$,    Lemma   \ref{book} implies  that      $D(u_{\ell-2})\cap D(v_{\ell-2})=\varnothing$.

Take  an arbitrary  maximal family   $\mathcal{P}$ of  mutually   vertex-disjoint  paths  of length $2\ell-4$ on degree-two  vertices of $G$ and call  $S_0$  the set of  all  middle  vertices of the   paths from   $\mathcal{P}$.   Extend $S_0$  to a   maximal subset $S$ of $V(G)$ in which  $D(x)\cap D(y)=\varnothing$   for any two distinct   vertices $x, y\in S$. Set $M=\bigcup_{x\in S} D(x)$.

\begin{lemma}\label{M2k}
$|M|=O(1)$.
\end{lemma}

\begin{proof}
As $G$ is   $C_{2\ell}$-saturated, there should be an edge between    $D(x)$ and  $D(y)$ for any  two distinct   vertices $x, y\in S$. This  implies that    $|E(G[M])|\geqslant{|S|\choose 2}$. On the other hand,     $|V(G[M])|< c^\ell|S|$. Recall  the fact  that any  $C_{2\ell}$-free graph on $n$ vertices has at most     $100\ell n^{1+1/\ell}$ edges    for all   integers $n$ sufficiently large \cite{bndy}. Since  $G[M]$ is  a  $C_{2\ell}$-free graph,   we get   that $|S|=O(1)$ and so $|M|=O(1)$.
\end{proof}

Let $A$ be the set of all degree-one vertices in   $V(G)\setminus M$  each of which is  adjacent to a vertex     in  $M$.  It follows  from    Lemma \ref{neigborone}(i)  that    $|A|\leqslant |M|$ and thus  $|A|=O(1)$    by Lemma  \ref{M2k}.
Let $B$ be the set of all vertices in  $V(G)\setminus M$  of degree at leat    $c$. We note that $M\cup B\not= \varnothing$.
Let $D_0$ be the set of all vertices in   $V(G)\setminus (M \cup   A\cup B)$ each of which has a neighbor      in   $M$.
For $i=1, 2, \ldots$,   denote   by $D_i$   the   set of all vertices in   $V(G)\setminus (M \cup   A\cup B\cup D_0)$    with distance $i$ from $B\cup D_0$.
Using  the maximality of $S$, $D_{\ell-1}=\varnothing$ and so the sets  $M,  A,  B,  D_0, D_1, \ldots, D_{\ell-2}$   partition    $V(G)$.
For any integer $i\in\{0, 1,   \ldots, \ell-2\}$   and  vertex $v\in D_i$, designate  an edge $e$ such that $e$ is incident with $v$ and the other endpoint of $e$ is in $B\cup D_{i-1}$,  by letting $D_{-1}=M$.
We color these designated edges black and all remaining edges  of $G$ gray.
For every gray  edge $e=xy$, we consider two gray half-edges $(x, e)$ and $(y, e)$ and we assign   edge-weight $1/2$ to each gray half-edge.
We call  each   vertex  in $B\cup D_0$ as  a  `root'  and  we     set $D=D_1\cup\cdots\cup D_{\ell-2}$.

Roughly speaking, our goal is to distribute the value $|E(G)|$ on  vertices in  $B\cup D_0\cup D$  such that each vertex in $B\cup D_0\cup D$   takes an edge-weight at least   $1+\alpha$.
This along  with Lemma \ref{M2k}    shows  that  $|E(G)|\geqslant (1+\alpha)|B\cup D_0\cup D|=(1+\alpha)|V(G)|+O(1)$, as desired.
For each  vertex in $D_0\cup D$, an edge-weight equal to $1$ comes from its designated  black edge. So,  we need   to add an extra   edge-weight equal to  $\alpha$  to each vertex in $D_0\cup D$.
For each vertex in $B$, we need to assign an edge-weight equal to  $1+\alpha$.
Since we have consumed all black edges, we will focus on   gray  edges hereafter.

For  convenience,  we  fix    some   notation   and terminology  below.
For any   subgraph  $H$   of $G$ and      vertex  $v\in V(H)$ with $\deg_H(v)=1$,
we say $v$ to be an  `$H$-pendant' vertex.    For each vertex $v\in V(G)$,  we define
$$\epsilon(v)=\left\{\begin{array}{ll}
0 & \mbox{ if } \deg_G(v)=1 \mbox{;} \\
1 & \mbox{ if } \deg_G(v)\geqslant2  \mbox{ and $v$ has no  $G$-pendant        neighbors;} \\
2 &  \mbox{ otherwise,}
\end{array}\right.$$
and  we   denote by    $g(v)$ the number of  gray half-edges  incident  with   $v$.

For every vertex $v\in D$, we introduce the following notation   and definitions.
Obviously, there exists a unique path from $v$ to  $D_1$ on black edges. We denote the set of vertices  of this path by $P_v$.
We have  $|P_v|\leqslant \ell-2$. We define the   tree $T_v$ as follows:
$V(T_v)$  is the set of all vertices  $x$ with  $v\in P_x$ and
$E(T_v)$  is the set of all black edges   both of whose endpoints are in  $V(T_v)$.
We define  $$g(T_v)=\sum_{x\in V(T_v)}g(x).$$ Clearly,
\begin{equation}\label{grelation}
g(T_v)=g(v)+\sum_{x\in N_{T_v}(v)} g(T_x).
\end{equation}
For each  gray half-edge incident with $v$, we distribute its edge-weight $1/2$ over the vertices in $P_v$ by  giving   the value $\alpha$ to each vertex of $P_v$ and  by  giving   what remains from $1/2$ to    $v$.
Note that this is possible, as  $1/2=(\ell-2)\alpha$.
The total edge-weight which $v$ gets from all  gray half-edges in this way is denoted by $\mathrm{wt}(v)$. If $v\in D_i$ for some integer $i\in\{1, \ldots, \ell-2\}$, then,        by the definition,
\begin{equation}\label{defwt}
\mathrm{wt}(v)=\big(g(T_v)+(\ell-2-i)g(v)\big)\alpha.
\end{equation}
Finally, we   define  the  overweight function $\beta$  as
\begin{equation}\label{defbeta}
\beta(T_v)=\sum_{x\in V(T_v)} \big(\mathrm{wt}(x)-\alpha\big).
\end{equation}
Notice that the values of   $\beta$  are all  integer multiples of $ \alpha$.

Below, we determine  all  trees   $T_v$ with $g(T_v)\in\{0, 1\}$.  Note that the converse  of each    statement  in the following Lemma  is   obviously  true.

\begin{lemma}\label{lCv}
Let $v\in D$. The following statements hold.
\begin{itemize}[noitemsep,  topsep=0pt]
\item[{\rm (i)}]  Let  $g(T_v)=0$. Then,    $v$ is a $G$-pendant   vertex.
\item[{\rm (ii)}]    Let  $g(T_v)=1$. Then,     $T_v$ is a path  and the unique  gray half-edge is incident with an endpoint of $T_v$. In addition,   the unique  gray half-edge is incident with $v$ if and only if either $V(T_v)=\{v\}$ or $V(T_v)=\{v, w\}$ for some  $G$-pendant   vertex   $w$.
\end{itemize}
\end{lemma}

\begin{proof}
Let $v\in D_i$ for some integer $i\in\{1, \ldots, \ell-2\}$.

\hspace{-3.7mm} (i)   We use backward  induction on $i$.    If $i=\ell-2$, then there is  nothing to prove.
Let $i\leqslant \ell-3$ and suppose toward a contradiction that  there is a vertex  $w\in N_{T_v}(v)$. We find from \eqref{grelation} that  $g(T_w)=0$  and  so   the backward  induction hypothesis    implies that $w$ is  a   $G$-pendant   vertex. This  along with  Lemma \ref{neigborone}(i)  yields   that  $V(T_v)=\{v, w\}$, meaning  that $v$ is a degree-two vertex of $G$ which is adjacent to the  $G$-pendant   vertex   $w$, contradicting Lemma  \ref{neigborone}(i).
This shows that $ N_{T_v}(v)=\varnothing$ which means that $\deg_G(v)=1$.

\hspace{-3.7mm} (ii)    We use backward  induction on $i$.    If $i=\ell-2$, then there is  nothing to prove.
Let $i\leqslant \ell-3$.
First, assume that  $g(v)=1$.  We have from   \eqref{grelation} that      $g(T_x)=0$  for every vertex  $x\in N_{T_v}(v)$. So, it follows from   Part    (i)   that  each  vertex  in $N_{T_v}(v)$    is  a  $G$-pendant   vertex   and thus     Lemma \ref{neigborone}(i) yields  that  $|N_{T_v}(v)|\leqslant1$, we are done.
Next,  assume that $g(v)=0$.  We deduce  from   \eqref{grelation} that  there is a vertex $w\in N_{T_v}(v)$  with      $g(T_w)=1$ and      $g(T_x)=0$  for every vertex  $x\in N_{T_v}(v)\setminus \{w\}$.   By   Part    (i),    $\deg_G(x)=1$  for every vertex  $x\in N_{T_v}(v)\setminus \{w\}$ and
by  the backward  induction  hypothesis, either $\deg_G(w)=2$ or $\deg_G(w)=3$ and $w$ has  a   $G$-pendant   neighbor.   If $\deg_G(w)=2$, then    Lemma \ref{neigborone}(i)  implies that  $v$ is not adjacent to a  $G$-pendant   vertex    and hence  $\deg_G(v)=2$, we are done.
If $\deg_G(w)=3$ and  $w$ has  a   $G$-pendant   neighbor, then    Lemma \ref{neigborone}(i)  yields   that  $\deg_G(v)\geqslant3$ and   $v$      has  no    $G$-pendant   neighbors, a contradiction.
\end{proof}

The following lemma presents  a lower bound for edge-weights of vertices in $D$ as well as  an applicable expression for   $\beta$.
Part (ii) of the following lemma shows that we assigned   an extra   edge-weight equal to  $\alpha$  to every    vertex in $D\setminus\{v\in D_1 \, | \, \deg_G(v)=1\}$.
It also  demonstrates   that     $\beta(T_v)\geqslant0$   if  $\deg_G(v)\geqslant2$. Note that it follows from  \eqref{defwt} and  \eqref{defbeta} that  $\beta(T_v)=-\alpha$   if  $\deg_G(v)=1$.

\begin{lemma}\label{lCvSH}
Let $v\in D$. The following statements hold.
\begin{itemize}[noitemsep,  topsep=0pt]
\item[{\rm (i)}]   $\mathrm{wt}(v)\geqslant (\deg_G(v)-\epsilon(v))\epsilon(v)\alpha$.
\item[{\rm (ii)}] $\mathrm{wt}(v)-\epsilon(v)\alpha\geqslant0$ and, if $\deg_G(v)\geqslant2$, then
\begin{equation}\label{betaN}\beta(T_v)= \sum_{x\in V(T_v)} \big(\mathrm{wt}(x)-\epsilon(x)\alpha\big).\end{equation}
\end{itemize}
\end{lemma}

\begin{proof}  Let $v\in D_i$ for some integer $i\in\{1, \ldots, \ell-2\}$.

\hspace{-3.7mm} (i)   The assertion is trivial if $\epsilon(v)=0$.
If $\epsilon(v)=1$,  then we   find  from  (\ref{grelation}), \eqref{defwt},  and Lemma \ref{lCv} that
$\mathrm{wt}(v)\geqslant g(T_v)\alpha \geqslant(g(v)+\deg_{T_v}(v))\alpha=(\deg_G(v)-1)\alpha$.
Hence,  we may assume that  $\epsilon(v)=2$. Let  $w$ be the $G$-pendant neighbor of $v$.  For each  vertex $x\in N_{T_v}(v)\setminus \{w\}$,  $x$ is not a $G$-pendent vertex by Lemma \ref{neigborone}(i)  and so
$g(T_x)\geqslant 1$ by Lemma \ref{lCv}.
If $g(T_x)= 1$ for some  vertex $x\in N_{T_v}(v)\setminus \{w\}$,  then either $\deg_G(x)=2$ or $\deg_G(x)=3$ and $x$ has  a   $G$-pendant   neighbor by Lemma \ref{lCv}.  However, both cases are impossible due to Lemma \ref{neigborone}(i).
Therefore,   $g(T_x)\geqslant 2$ for each  vertex $x\in N_{T_v}(v)\setminus \{w\}$.
Moreover,    since  $v$ has a $G$-pendant neighbor,  $v\not\in D_{\ell-2}$ and so $i\leqslant\ell-3$.
Now, we   obtain   from  (\ref{grelation}) and  \eqref{defwt}  that
$\mathrm{wt}(v)\geqslant (g(T_v)+g(v))\alpha \geqslant2(g(v)+\deg_{T_v}(v)-1)\alpha=2(\deg_G(v)-2)\alpha$.

\hspace{-3.7mm} (ii)  Note first that, if  $\epsilon(v)=2$, then $\deg_G(v)\geqslant3$ by Lemma \ref{neigborone}(i). From  this and   Part (i), we get that  $\mathrm{wt}(v)-\epsilon(v)\alpha\geqslant(\deg_G(v)-\epsilon(v)-1)\epsilon(v)\alpha\geqslant0$.
Now, let      $\deg_G(v)\geqslant2$.
The contribution
of  each    vertex      $x\in V(T_v)$ with $\epsilon(x)=1$  to the   summation in  \eqref{betaN} is trivially  the  same as to   the   summation in   \eqref{defbeta}.
But,  the contribution
of  each   $G$-pendant vertex      $w\in V(T_v)$
to the   summation in   \eqref{betaN}  is $0$  while its contribution  to the   summation in  \eqref{defbeta}  is $-\alpha$.
It is  compensated by  the contribution of the  unique   neighbor of $w$, say $u$. Indeed, the  contribution of $u$  to
the   summation in   \eqref{betaN}  is $\mathrm{wt}(u)-2\alpha$   while its contribution  to the   summation in  \eqref{defbeta}  is $\mathrm{wt}(u)-\alpha$.
This demonstrates  that \eqref{betaN} holds,  since    distinct $G$-pendant vertices have distinct neighbors by Lemma \ref{neigborone}(i).
\end{proof}

In the following lemma, we determine  all  trees   $T_u$ with $u  \in D_1\cup D_2 \cup D_3 \cup  D_4$ and  $\beta(T_u)=0$.

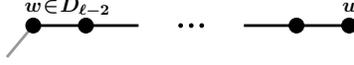
\begin{figure}
\centering
\scalebox{0.7}{\boldmath{
\begin{tikzpicture}
\begin{pgfonlayer}{nodelayer}
\draw [fill=black] (-4, 3.75) circle (4pt);
\draw [fill=black] (-5, 3.75) circle (4pt);
\draw [fill=black] (-9, 3.75) circle (4pt);
\draw [fill=black] (-10, 3.75) circle (4pt);
\draw [fill=black] (-7.2, 3.75) circle (1pt);
\draw [fill=black] (-7, 3.75) circle (1pt);
\draw [fill=black] (-6.8, 3.75) circle (1pt);
\node [style=none] (0) at (-4, 3.75) {};
\node [style=none] (1) at (-5, 3.75) {};
\node [style=none] (2) at (-6, 3.75) {};
\node [style=none] (3) at (-7, 3.75) {};
\node [style=none] (4) at (-8, 3.75) {};
\node [style=none] (5) at (-9, 3.75) {};
\node [style=none] (6) at (-10, 3.75) {};
\node [style=none] (7) at (-10.5, 3.15) {};
\node [style=none] (8) at (-9.35, 4.1) {$w \hspace{-1mm}\in  \hspace{-1mm} D_{\ell -2}$};
\node [style=none] (9) at (-4, 4.1) {$u$};
\end{pgfonlayer}
\begin{pgfonlayer}{edgelayer}
\draw [line width=1.5pt,color=gray!80,fill=gray!80] (7.center) to (6.center);
\draw [line width=1.5pt,color=black,fill=black] (6.center) to (5.center);
\draw [line width=1.5pt,color=black,fill=black]  (5.center) to (4.center);
\draw [line width=1.5pt,color=black,fill=black]  (2.center) to (1.center);
\draw [line width=1.5pt,color=black,fill=black]  (1.center) to (0.center);
\end{pgfonlayer}
\end{tikzpicture}}}
\caption{The tree  of Lemma \ref{beta0}.}\label{alpha0}
\end{figure}

\begin{lemma}\label{beta0}
Let  $u  \in D_1\cup D_2 \cup D_3 \cup  D_4$  and  $\beta(T_u)=0$. Then, $T_u$  is a path with a unique gray
half-edge which is incident with a $T_u$-pendant vertex belonging to $D_{\ell-2}$    as  depicted  in Figure \ref{alpha0}.
\end{lemma}

\begin{proof}
For  each   vertex $x\in V(T_u)$,     Lemma \ref{lCvSH}  yields  that   $\mathrm{wt}(x)-\epsilon(x)\alpha=0$ and $\epsilon(x)\alpha=\mathrm{wt}(x)\geqslant (\deg_G(x)-\epsilon(x))\epsilon(x)\alpha$ which implies  that $\deg_G(x)\leqslant \epsilon(x)+1$.  This means that,  for each    vertex $x\in V(T_u)$,   either $\deg_G(x)\in\{1, 2\}$ or   $\deg_G(x)=3$ and   $\epsilon(x)=2$.
By  contradiction, suppose that there is a vertex  $v\in V(T_u)$ with   $\deg_G(v)=3$ and   $\epsilon(v)=2$. Letting    $v'\in N_{T_u}(v)$ be $G$-pendant,   Lemma \ref{neigborone}(i)  forces that  $N_{T_u}(v)\setminus\{v'\}=\varnothing$. Thus,   $v=u$  and   $T_u$ is a path of length $1$   with a unique  gray half-edge which is  incident with $u$. However,   it follows from $\beta(T_u)=0$ that $u\in D_{\ell-3}$ which contradicts $\ell\geqslant14$.  This contradiction shows that     $\deg_G(x)\in\{1, 2\}$  for every    vertex $x\in V(T_u)$, meaning  that      $T_u$ is a path.   From $\beta(T_u)=0$, $T_u$ is not a single-vertex tree      and thus   $g(T_u)\geqslant1$ using      Lemma \ref{lCv}(i).  Therefore,    $T_u$ has a unique  gray half-edge which is    incident with a $T_u$-pendant vertex $w$ other than $u$.  As $\beta(T_u)=0$, we conclude that  $w\in D_{\ell-2}$.
\end{proof}

In the following lemma, we determine  all  trees   $T_u$ with $u\in   D_2\cup D_3$ and  $\beta(T_u)=\alpha$.

\begin{figure}
\centering
\scalebox{0.7}{\boldmath{
\begin{tikzpicture}
\begin{pgfonlayer}{nodelayer}
\draw [fill=black] (-4, 3.75) circle (4pt);
\draw [fill=black] (-5, 3.75) circle (4pt);
\draw [fill=black] (-9, 3.75) circle (4pt);
\draw [fill=black] (-10, 3.75) circle (4pt);
\draw [fill=black] (-7.2, 3.75) circle (1pt);
\draw [fill=black] (-7, 3.75) circle (1pt);
\draw [fill=black] (-6.8, 3.75) circle (1pt);
\node [style=none] (0) at (-4, 3.75) {$\bullet$};
\node [style=none] (1) at (-5, 3.75) {$\bullet$};
\node [style=none] (2) at (-6, 3.75) {};
\node [style=none] (3) at (-7, 3.75) {};
\node [style=none] (4) at (-8, 3.75) {};
\node [style=none] (5) at (-9, 3.75) {$\bullet$};
\node [style=none] (6) at (-10, 3.75) {$\bullet$};
\node [style=none] (7) at (-10.5, 3.25) {};
\node [style=none] (8) at (-9.35, 4.1) {$w \hspace{-1mm}\in  \hspace{-1mm} D_{\ell -3}$};
\node [style=none] (9) at (-4, 4.1) {$u$};
\node [style=none] (11) at (-7, 0.75) {(i)};
\draw [fill=black] (3.5, -3) circle (4pt);
\draw [fill=black] (4.5, -2) circle (4pt);
\draw [fill=black] (2.5, -2.75) circle (4pt);
\draw [fill=black] (-1.5, -1.75) circle (4pt);
\draw [fill=black] (-2.5, -1.5) circle (4pt);
\draw [fill=black] (-1.5, -4.25) circle (4pt);
\draw [fill=black] (2.5, -3.25) circle (4pt);
\draw [fill=black] (-2.5, -4.5) circle (4pt);
\draw [fill=black] (4.5, -3) circle (4pt);
\draw [fill=black] (0.5, -2.25) circle (1pt);
\draw [fill=black] (0.7, -2.3) circle (1pt);
\draw [fill=black] (0.3, -2.2) circle (1pt);
\draw [fill=black] (0.5, -3.75) circle (1pt);
\draw [fill=black] (0.7, -3.7) circle (1pt);
\draw [fill=black] (0.3, -3.8) circle (1pt);
\node [style=none] (40) at (3.5, -3) {$\bullet$};
\node [style=none] (41) at (4.5, -2) {$\bullet$};
\node [style=none] (42) at (0.5, -6) {(iii)};
\node [style=none] (43) at (2.5, -2.75) {$\bullet$};
\node [style=none] (44) at (1.5, -2.5) {};
\node [style=none] (46) at (-0.5, -2) {};
\node [style=none] (47) at (-1.5, -1.75) {$\bullet$};
\node [style=none] (48) at (-2.5, -1.5) {$\bullet$};
\node [style=none] (49) at (2.5, -3.25) {$\bullet$};
\node [style=none] (50) at (1.5, -3.5) {};
\node [style=none] (51) at (-3, -2) {};
\node [style=none] (53) at (-0.5, -4) {};
\node [style=none] (54) at (-1.5, -4.25) {$\bullet$};
\node [style=none] (55) at (-2.5, -4.5) {$\bullet$};
\node [style=none] (56) at (-3, -5) {};
\node [style=none] (57) at (-1.8, -1.15) {$w_{1}  \hspace{-1mm} \in \hspace{-1mm} D_{\ell -2}$};
\node [style=none] (58) at (-1.7, -4.8) {$w_{2}  \hspace{-1mm} \in \hspace{-1mm} D_{\ell -2}$};
\node [style=none] (60) at (3.5, -3) {};
\draw [fill=black] (11, 3.75) circle (4pt);
\draw [fill=black] (10, 4) circle (4pt);
\draw [fill=black] (6, 5) circle (4pt);
\draw [fill=black] (5, 5.25) circle (4pt);
\draw [fill=black] (10, 3.5) circle (4pt);
\draw [fill=black] (6, 2.5) circle (4pt);
\draw [fill=black] (5, 2.25) circle (4pt);
\draw [fill=black] (8, 4.5) circle (1pt);
\draw [fill=black] (8.2, 4.45) circle (1pt);
\draw [fill=black] (7.8, 4.55) circle (1pt);
\draw [fill=black] (8, 3) circle (1pt);
\draw [fill=black] (8.2, 3.05) circle (1pt);
\draw [fill=black] (7.8, 2.95) circle (1pt);
\node [style=none] (61) at (11, 3.75) {$\bullet$};
\node [style=none] (62) at (11, 4.1) {$u$};
\node [style=none] (63) at (8, 0.75) {(ii)};
\node [style=none] (64) at (10, 4) {$\bullet$};
\node [style=none] (65) at (9, 4.25) {};
\node [style=none] (67) at (7, 4.75) {};
\node [style=none] (68) at (6, 5) {$\bullet$};
\node [style=none] (69) at (5, 5.25) {$\bullet$};
\node [style=none] (70) at (10, 3.5) {$\bullet$};
\node [style=none] (71) at (9, 3.25) {};
\node [style=none] (72) at (4.5, 4.75) {};
\node [style=none] (74) at (7, 2.75) {};
\node [style=none] (75) at (6, 2.5) {$\bullet$};
\node [style=none] (76) at (5, 2.25) {$\bullet$};
\node [style=none] (77) at (4.5, 1.75) {};
\node [style=none] (78) at (5.7, 5.6) {$w_{1} \hspace{-1mm} \in \hspace{-1mm} D_{\ell -2}$};
\node [style=none] (79) at (5.8, 1.9) {$w_{2}\hspace{-1mm} \in \hspace{-1mm} D_{\ell -2}$};
\node [style=none] (81) at (11, 3.75) {};
\node [style=none] (82) at (4.5, -3) {$\bullet$};
\node [style=none] (85) at (4.5, -1.65) {$v$};
\node [style=none] (86) at (3.6, -3.3) {$w$};
\node [style=none] (88) at (4.5, -3.3) {$u$};
\end{pgfonlayer}
\begin{pgfonlayer}{edgelayer}
\draw [line width=1.5pt,color=gray!80,fill=gray!80] (7.center) to (6.center);
\draw [line width=1.5pt,color=black,fill=black] (6.center) to (5.center);
\draw [line width=1.5pt,color=black,fill=black]  (5.center) to (4.center);
\draw [line width=1.5pt,color=black,fill=black]  (2.center) to (1.center);
\draw [line width=1.5pt,color=black,fill=black]  (1.center) to (0.center);
\draw [line width=1.5pt,color=black,fill=black] (40.center) to (43.center);
\draw [line width=1.5pt,color=black,fill=black] (43.center) to (44.center);
\draw [line width=1.5pt,color=black,fill=black] (46.center) to (47.center);
\draw [line width=1.5pt,color=black,fill=black] (47.center) to (48.center);
\draw [line width=1.5pt,color=gray!80,fill=gray!80](48.center) to (51.center);
\draw [line width=1.5pt,color=black,fill=black](40.center) to (49.center);
\draw [line width=1.5pt,color=black,fill=black](49.center) to (50.center);
\draw [line width=1.5pt,color=black,fill=black](53.center) to (54.center);
\draw [line width=1.5pt,color=black,fill=black](54.center) to (55.center);
\draw [line width=1.5pt,color=gray!80,fill=gray!80](55.center) to (56.center);
\draw [line width=1.5pt,color=black,fill=black](61.center) to (64.center);
\draw [line width=1.5pt,color=black,fill=black](64.center) to (65.center);
\draw [line width=1.5pt,color=black,fill=black](67.center) to (68.center);
\draw [line width=1.5pt,color=black,fill=black](68.center) to (69.center);
\draw [line width=1.5pt,color=gray!80,fill=gray!80](69.center) to (72.center);
\draw [line width=1.5pt,color=black,fill=black](61.center) to (70.center);
\draw [line width=1.5pt,color=black,fill=black](70.center) to (71.center);
\draw [line width=1.5pt,color=black,fill=black](74.center) to (75.center);
\draw [line width=1.5pt,color=black,fill=black](75.center) to (76.center);
\draw [line width=1.5pt,color=gray!80,fill=gray!80](76.center) to (77.center);
\draw [line width=1.5pt,color=black,fill=black](60.center) to (82.center);
\draw [line width=1.5pt,color=black,fill=black](82.center) to (41.center);
\end{pgfonlayer}
\end{tikzpicture}}}
\caption{The trees  of Lemma \ref{beta1}.}\label{alpha}
\end{figure}
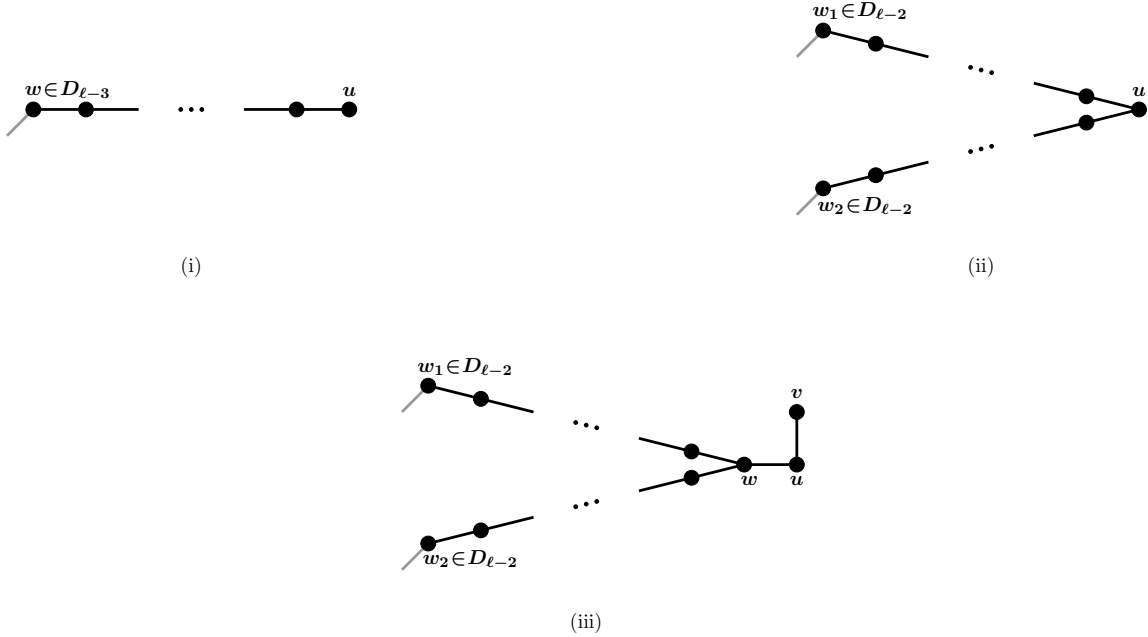

\begin{lemma}\label{beta1}
Let $u\in   D_2\cup D_3$   and  $\beta(T_u)=\alpha$.     Then,   $T_u$ is one  of the  three  trees illustrated  in Figure \ref{alpha}, where all gray half-edges incident with $T_u$ are also shown.
\end{lemma}

\begin{proof}
As  $\beta(T_u)=\alpha$, $T_u$ is not a single-vertex tree      and  so     $g(T_u)\geqslant1$   by    Lemma \ref{lCv}(i).
If  $g(u)\geqslant1$,  then we obtain   from       Lemma \ref{lCvSH}(ii),   \eqref{defwt},  and $\ell\geqslant14$    that
$$\alpha=\beta(T_u)\geqslant \mathrm{wt}(u)-\epsilon(u)\alpha\geqslant \big(g(T_u)+(\ell-2-3)g(u)-\epsilon(u)\big)\alpha\geqslant(\ell-6)\alpha,$$  a contradiction.  Hence,    $g(u)=0$.
If $g(T_u)=1$, then     Lemma \ref{lCv}(ii)  and  $\beta(T_u)=\alpha$  yield    that    $T_u$ is a path  and the   unique  gray half-edge is incident with a $T_u$-pendant   vertex   $w\in D_{\ell-3}$ other than $u$. In this case,  $T_u$ has the  shape    illustrated  in Figure \ref{alpha}(i).  So, in what follows, we let  $g(T_u)\geqslant2$.

First, assume that   $\mathrm{wt}(u)-\epsilon(u)\alpha=\alpha$. By applying  Lemma \ref{lCvSH}(ii)   for any vertex  $x\in N_{T_u}(u)$ with $\deg_G(x)\geqslant2$, we conclude that $\beta(T_x)=0$. This along with  Lemma \ref{beta0} implies that
$T_u-u$ is a vertex-disjoint   union of  some paths, each of which  has   a unique  gray half-edge  incident with  its endpoint   belonging to $D_{\ell-2}$.
Moreover, it follows from Lemma \ref{neigborone}(i)  that $\epsilon(u)=1$. By   $\mathrm{wt}(u)=2\alpha$,  $g(u)=0$,  and  \eqref{defwt},   we  conclude     that  $g(T_u)=2$ and so
$\deg_{T_u}(u)=2$. Thus,     in this case,    $T_u$ has the  shape   depicted  in Figure \ref{alpha}(ii).

Next, assume that   $\mathrm{wt}(u)-\epsilon(u)\alpha=0$.  By     $g(T_u)\geqslant2$,  $g(u)=0$,  and  \eqref{defwt},   we  find      that  $g(T_u)=2$ and $\epsilon(u)=2$. Letting $v\in N_{T_u}(u)$ be    $G$-pendant, Lemma \ref{neigborone}(i)  implies that $\deg_G(x)\geqslant3$ for each vertex   $x\in N_{T_u}(u)\setminus\{v\}$.    This  along with  $g(T_u)=2$ and  Lemma \ref{lCv}(ii) implies  that $\deg_{T_u}(u)=2$.
Let $N_{T_u}(u)=\{v,w\}$ for some vertex $w$.  We have from \eqref{betaN} that  $\beta(T_{w})=\alpha$. Also,  by    $g(u)=0$, $g(T_u)=2$,   and   \eqref{grelation}, we deduce    that   $g(T_{w})=2$. Moreover,  it follows   from
$\deg_G(u)=3$ and  Lemma \ref{neigborone}(ii)  that   $\epsilon(w)=1$.
Using    Lemma \ref{lCvSH}(ii) and  \eqref{defwt}, we derive   that    $$\alpha=\beta(T_{w})\geqslant\mathrm{wt}(w)-\epsilon(w)\alpha\geqslant\alpha+ (\ell-6)g(w)\alpha.$$   As $\ell\geqslant14$,  we conclude  that    $g(w)=0$. Since  $\deg_G(u)=3$,   $\epsilon(w)=1$, and $g(T_{w})=2$, we   obtain from    Lemma \ref{lCv}(ii)  and  $\beta(T_{w})=\alpha$       that   $T_{w}-w$ is a vertex-disjoint   union of  two  paths, each of which  has   a unique  gray half-edge  incident with  its endpoint   belonging to $D_{\ell-2}$.
Hence,   in this case,    $T_u$ has the  shape  indicated   in Figure \ref{alpha}(iii).
\end{proof}

In the following lemma, we determine  all  trees   $T_u$ with $u\in   D_2$ and  $\beta(T_u)=2\alpha$.

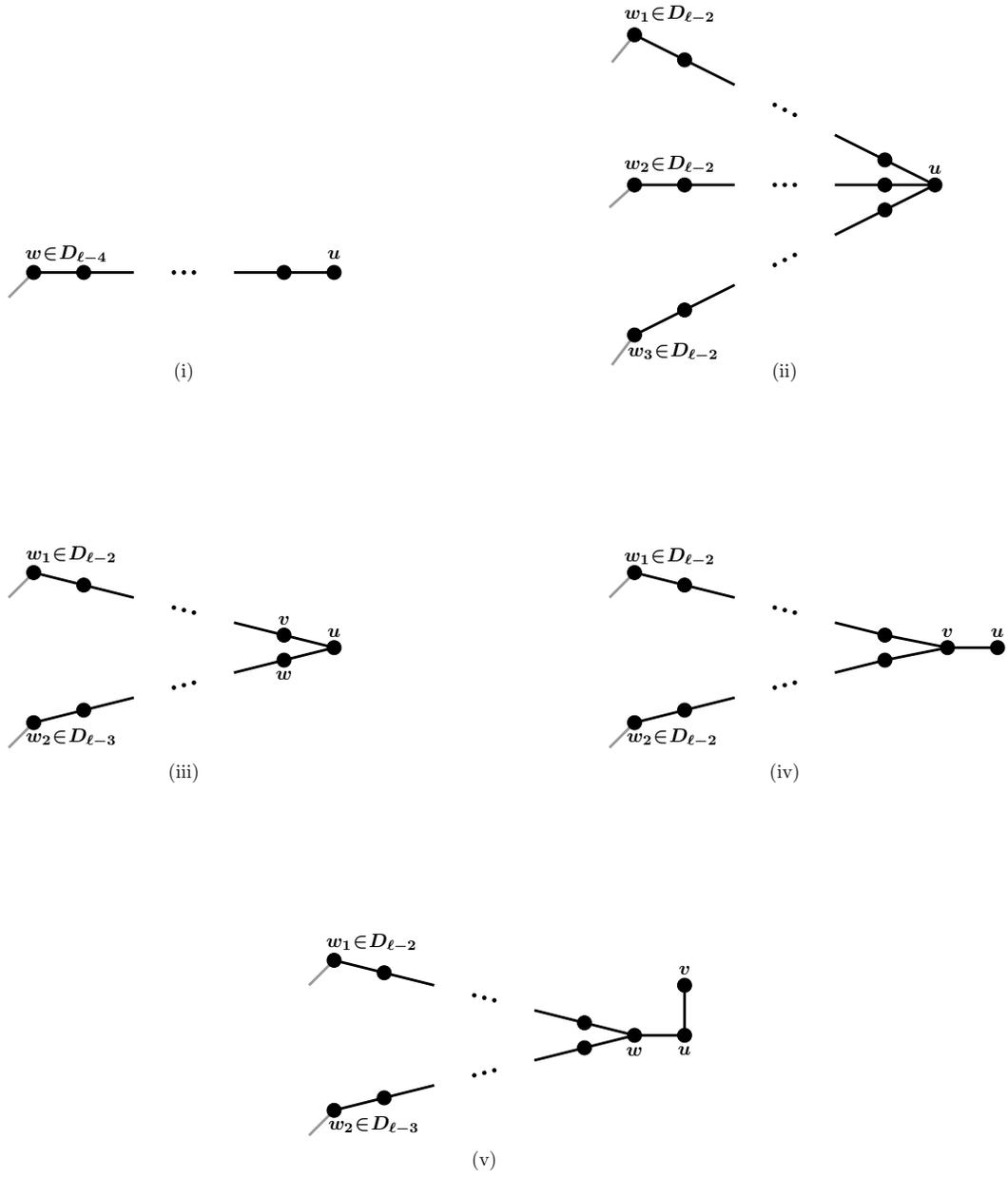
\begin{figure}
\centering
\scalebox{0.7}{\boldmath{
\begin{tikzpicture}
\begin{pgfonlayer}{nodelayer}
\draw [fill=black] (-4.5, 7) circle (4pt);
\draw [fill=black] (-5.5, 7) circle (4pt);
\draw [fill=black] (-9.5, 7) circle (4pt);
\draw [fill=black] (-10.5, 7) circle (4pt);
\draw [fill=black] (-7.5, 7) circle (1pt);
\draw [fill=black] (-7.3, 7) circle (1pt);
\draw [fill=black] (-7.7, 7) circle (1pt);
\node [style=none] (0) at (-4.5, 7) {$\bullet$};
\node [style=none] (1) at (-5.5, 7) {$\bullet$};
\node [style=none] (2) at (-6.5, 7) {};
\node [style=none] (3) at (-7.5, 7) {};
\node [style=none] (4) at (-8.5, 7) {};
\node [style=none] (5) at (-9.5, 7) {$\bullet$};
\node [style=none] (6) at (-10.5, 7) {$\bullet$};
\node [style=none] (7) at (-11, 6.5) {};
\node [style=none] (8) at (-9.85, 7.35) {$w\hspace{-1mm} \in \hspace{-1mm} D_{\ell -4}$};
\node [style=none] (9) at (-4.5, 7.35) {$u$};
\node [style=none] (10) at (-7.5, 7) {};
\node [style=none] (11) at (-7.5, 5) {(i)};
\draw [fill=black] (7.5, 8.75) circle (4pt);
\draw [fill=black] (6.5, 8.75) circle (4pt);
\draw [fill=black] (2.5, 8.75) circle (4pt);
\draw [fill=black] (1.5, 8.75) circle (4pt);
\draw [fill=black] (6.5, 9.25) circle (4pt);
\draw [fill=black] (2.5, 11.25) circle (4pt);
\draw [fill=black] (1.5, 11.75) circle (4pt);
\draw [fill=black] (6.5, 8.25) circle (4pt);
\draw [fill=black] (1.5, 5.75) circle (4pt);
\draw [fill=black] (2.5, 6.25) circle (4pt);
\draw [fill=black] (4.5, 10.25) circle (1pt);
\draw [fill=black] (4.3, 10.35) circle (1pt);
\draw [fill=black] (4.7, 10.15) circle (1pt);
\draw [fill=black] (4.5, 8.75) circle (1pt);
\draw [fill=black] (4.3, 8.75) circle (1pt);
\draw [fill=black] (4.7, 8.75) circle (1pt);
\draw [fill=black] (4.5, 7.25) circle (1pt);
\draw [fill=black] (4.3, 7.15) circle (1pt);
\draw [fill=black] (4.7, 7.35) circle (1pt);
\node [style=none] (12) at (7.5, 8.75) {$\bullet$};
\node [style=none] (13) at (6.5, 8.75) {$\bullet$};
\node [style=none] (14) at (5.5, 8.75) {};
\node [style=none] (15) at (4.5, 8.75) {};
\node [style=none] (16) at (3.5, 8.75) {};
\node [style=none] (17) at (2.5, 8.75) {$\bullet$};
\node [style=none] (18) at (1.5, 8.75) {$\bullet$};
\node [style=none] (19) at (1, 8.3) {};
\node [style=none] (20) at (2.2, 9.15) {$w_{2} \hspace{-1mm}\in \hspace{-1mm} D_{\ell -2}$};
\node [style=none] (21) at (7.5, 9.05) {$u$};
\node [style=none] (22) at (4.5, 8.75) {};
\node [style=none] (23) at (4.5, 5) {(ii)};
\node [style=none] (24) at (6.5, 9.25) {$\bullet$};
\node [style=none] (25) at (5.5, 9.75) {};
\node [style=none] (26) at (4.5, 10.25) {};
\node [style=none] (27) at (3.5, 10.75) {};
\node [style=none] (28) at (2.5, 11.25) {$\bullet$};
\node [style=none] (29) at (1.5, 11.75) {$\bullet$};
\node [style=none] (30) at (6.5, 8.25) {$\bullet$};
\node [style=none] (31) at (5.5, 7.75) {};
\node [style=none] (32) at (1.05, 11.2) {};
\node [style=none] (33) at (4.5, 7.25) {};
\node [style=none] (34) at (3.5, 6.75) {};
\node [style=none] (35) at (2.5, 6.25) {$\bullet$};
\node [style=none] (36) at (1.5, 5.75) {$\bullet$};
\node [style=none] (37) at (1.05, 5.15) {};
\node [style=none] (38) at (2.2, 12.15) {$w_{1} \hspace{-1mm} \in \hspace{-1mm} D_{\ell -2}$};
\node [style=none] (39) at (2.27, 5.4) {$w_{3} \hspace{-1mm} \in \hspace{-1mm} D_{\ell -2}$};
\draw [fill=black] (-4.5, -0.5) circle (4pt);
\draw [fill=black] (-5.5, -0.25) circle (4pt);
\draw [fill=black] (-9.5, 0.75) circle (4pt);
\draw [fill=black] (-10.5, 1) circle (4pt);
\draw [fill=black] (-5.5, -0.75) circle (4pt);
\draw [fill=black] (-9.5, -1.75) circle (4pt);
\draw [fill=black] (-10.5, -2) circle (4pt);
\draw [fill=black] (-7.5, 0.25) circle (1pt);
\draw [fill=black] (-7.7, 0.3) circle (1pt);
\draw [fill=black] (-7.3, 0.2) circle (1pt);
\draw [fill=black] (-7.5, -1.25) circle (1pt);
\draw [fill=black] (-7.7, -1.3) circle (1pt);
\draw [fill=black] (-7.3, -1.2) circle (1pt);
\node [style=none] (40) at (-4.5, -0.5) {$\bullet$};
\node [style=none] (41) at (-4.5, -0.2) {$u$};
\node [style=none] (42) at (-7.5, -3) {(iii)};
\node [style=none] (42) at (-5.5, 0.05) {$v$};
\node [style=none] (42) at (-5.5, -1.05) {$w$};
\node [style=none] (43) at (-5.5, -0.25) {$\bullet$};
\node [style=none] (44) at (-6.5, 0) {};
\node [style=none] (45) at (-7.5, 0.25) {};
\node [style=none] (46) at (-8.5, 0.5) {};
\node [style=none] (47) at (-9.5, 0.75) {$\bullet$};
\node [style=none] (48) at (-10.5, 1) {$\bullet$};
\node [style=none] (49) at (-5.5, -0.75) {$\bullet$};
\node [style=none] (50) at (-6.5, -1) {};
\node [style=none] (51) at (-11, 0.5) {};
\node [style=none] (52) at (-7.5, -1.25) {};
\node [style=none] (53) at (-8.5, -1.5) {};
\node [style=none] (54) at (-9.5, -1.75) {$\bullet$};
\node [style=none] (55) at (-10.5, -2) {$\bullet$};
\node [style=none] (56) at (-11, -2.5) {};
\node [style=none] (57) at (-9.75, 1.35) {$w_{1} \hspace{-1mm} \in \hspace{-1mm} D_{\ell -2}$};
\node [style=none] (58) at (-9.75, -2.3) {$w_{2} \hspace{-1mm} \in \hspace{-1mm} D_{\ell -3}$};
\node [style=none] (59) at (-4.5, -0.5) {};
\draw [fill=black] (7.75, -0.5) circle (4pt);
\draw [fill=black] (8.75, -0.5) circle (4pt);
\draw [fill=black] (6.5, -0.25) circle (4pt);
\draw [fill=black] (2.5, 0.75) circle (4pt);
\draw [fill=black] (1.5, 1) circle (4pt);
\draw [fill=black] (6.5, -0.75) circle (4pt);
\draw [fill=black] (2.5, -1.75) circle (4pt);
\draw [fill=black] (1.5, -2) circle (4pt);
\draw [fill=black] (4.5, 0.25) circle (1pt);
\draw [fill=black] (4.7, 0.2) circle (1pt);
\draw [fill=black] (4.3, 0.3) circle (1pt);
\draw [fill=black] (4.5, -1.25) circle (1pt);
\draw [fill=black] (4.7, -1.2) circle (1pt);
\draw [fill=black] (4.3, -1.3) circle (1pt);
\node [style=none] (60) at (7.75, -0.5) {$\bullet$};
\node [style=none] (61) at (4.5, -3) {(iv)};
\node [style=none] (62) at (2.2, 1.3) {$w_{1} \hspace{-1mm} \in \hspace{-1mm} D_{\ell -2}$};
\node [style=none] (63) at (2.25, -2.3) {$w_{2} \hspace{-1mm} \in \hspace{-1mm} D_{\ell -2}$};
\node [style=none] (64) at (8.75, -0.2) {$u$};
\node [style=none] (65) at (7.75, -0.2) {$v$};
\node [style=none] (66) at (8.75, -0.5) {$\bullet$};
\node [style=none] (67) at (6.5, -0.25) {$\bullet$};
\node [style=none] (68) at (5.5, 0) {};
\node [style=none] (69) at (4.5, 0.25) {};
\node [style=none] (70) at (3.5, 0.5) {};
\node [style=none] (71) at (2.5, 0.75) {$\bullet$};
\node [style=none] (72) at (1.5, 1) {$\bullet$};
\node [style=none] (73) at (6.5, -0.75) {$\bullet$};
\node [style=none] (74) at (5.5, -1) {};
\node [style=none] (75) at (1, 0.5) {};
\node [style=none] (76) at (4.5, -1.25) {};
\node [style=none] (77) at (3.5, -1.5) {};
\node [style=none] (78) at (2.5, -1.75) {$\bullet$};
\node [style=none] (79) at (1.5, -2) {$\bullet$};
\node [style=none] (80) at (1, -2.5) {};
\draw [fill=black] (1.5, -8.25) circle (4pt);
\draw [fill=black] (2.5, -8.25) circle (4pt);
\draw [fill=black] (2.5, -7.25) circle (4pt);
\draw [fill=black] (0.5, -8) circle (4pt);
\draw [fill=black] (-3.5, -7) circle (4pt);
\draw [fill=black] (-4.5, -6.75) circle (4pt);
\draw [fill=black] (0.5, -8.5) circle (4pt);
\draw [fill=black] (-3.5, -9.5) circle (4pt);
\draw [fill=black] (-4.5, -9.75) circle (4pt);
\draw [fill=black] (-1.5, -7.5) circle (1pt);
\draw [fill=black] (-1.3, -7.55) circle (1pt);
\draw [fill=black] (-1.7, -7.45) circle (1pt);
\draw [fill=black] (-1.5, -9) circle (1pt);
\draw [fill=black] (-1.3, -8.95) circle (1pt);
\draw [fill=black] (-1.7, -9.05) circle (1pt);
\node [style=none] (81) at (1.5, -8.25) {$\bullet$};
\node [style=none] (82) at (-1.5, -10.75) {(v)};
\node [style=none] (83) at (-3.75, -6.4) {$w_{1} \hspace{-1mm} \in \hspace{-1mm} D_{\ell -2}$};
\node [style=none] (84) at (-3.71, -10.05) {$w_{2} \hspace{-1mm} \in \hspace{-1mm} D_{\ell -3}$};
\node [style=none] (85) at (2.5, -8.55) {$u$};
\node [style=none] (86) at (1.5, -8.57) {$w$};
\node [style=none] (86) at (2.5, -6.95) {$v$};
\node [style=none] (87) at (2.5, -8.25) {$\bullet$};
\node [style=none] (88) at (2.5, -7.25) {$\bullet$};
\node [style=none] (89) at (0.5, -8) {$\bullet$};
\node [style=none] (90) at (-0.5, -7.75) {};
\node [style=none] (91) at (-1.5, -7.5) {};
\node [style=none] (92) at (-2.5, -7.25) {};
\node [style=none] (93) at (-3.5, -7) {$\bullet$};
\node [style=none] (94) at (-4.5, -6.75) {$\bullet$};
\node [style=none] (95) at (0.5, -8.5) {$\bullet$};
\node [style=none] (96) at (-0.5, -8.75) {};
\node [style=none] (97) at (-5, -7.25) {};
\node [style=none] (98) at (-1.5, -9) {};
\node [style=none] (99) at (-2.5, -9.25) {};
\node [style=none] (100) at (-3.5, -9.5) {$\bullet$};
\node [style=none] (101) at (-4.5, -9.75) {$\bullet$};
\node [style=none] (102) at (-5, -10.25) {};
\end{pgfonlayer}
\begin{pgfonlayer}{edgelayer}
\draw [line width=1.5pt,color=gray!80,fill=gray!80](7.center) to (6.center);
\draw [line width=1.5pt,color=black,fill=black](6.center) to (5.center);
\draw [line width=1.5pt,color=black,fill=black](5.center) to (4.center);
\draw [line width=1.5pt,color=black,fill=black](2.center) to (1.center);
\draw [line width=1.5pt,color=black,fill=black](1.center) to (0.center);
\draw [line width=1.5pt,color=gray!80,fill=gray!80](19.center) to (18.center);
\draw [line width=1.5pt,color=black,fill=black](18.center) to (17.center);
\draw [line width=1.5pt,color=black,fill=black](17.center) to (16.center);
\draw [line width=1.5pt,color=black,fill=black](14.center) to (13.center);
\draw [line width=1.5pt,color=black,fill=black](13.center) to (12.center);
\draw [line width=1.5pt,color=black,fill=black](12.center) to (24.center);
\draw [line width=1.5pt,color=black,fill=black](24.center) to (25.center);
\draw [line width=1.5pt,color=black,fill=black](27.center) to (28.center);
\draw [line width=1.5pt,color=black,fill=black](28.center) to (29.center);
\draw [line width=1.5pt,color=gray!80,fill=gray!80](29.center) to (32.center);
\draw [line width=1.5pt,color=black,fill=black](12.center) to (30.center);
\draw [line width=1.5pt,color=black,fill=black](30.center) to (31.center);
\draw [line width=1.5pt,color=black,fill=black](34.center) to (35.center);
\draw [line width=1.5pt,color=black,fill=black](35.center) to (36.center);
\draw [line width=1.5pt,color=gray!80,fill=gray!80] (36.center) to (37.center);
\draw [line width=1.5pt,color=black,fill=black](40.center) to (43.center);
\draw [line width=1.5pt,color=black,fill=black](43.center) to (44.center);
\draw [line width=1.5pt,color=black,fill=black](46.center) to (47.center);
\draw [line width=1.5pt,color=black,fill=black](47.center) to (48.center);
\draw [line width=1.5pt,color=gray!80,fill=gray!80](48.center) to (51.center);
\draw [line width=1.5pt,color=black,fill=black](40.center) to (49.center);
\draw [line width=1.5pt,color=black,fill=black](49.center) to (50.center);
\draw [line width=1.5pt,color=black,fill=black](53.center) to (54.center);
\draw [line width=1.5pt,color=black,fill=black](54.center) to (55.center);
\draw [line width=1.5pt,color=gray!80,fill=gray!80](55.center) to (56.center);
\draw [line width=1.5pt,color=black,fill=black](60.center) to (66.center);
\draw [line width=1.5pt,color=black,fill=black](67.center) to (68.center);
\draw [line width=1.5pt,color=black,fill=black](70.center) to (71.center);
\draw [line width=1.5pt,color=black,fill=black](71.center) to (72.center);
\draw [line width=1.5pt,color=gray!80,fill=gray!80](72.center) to (75.center);
\draw [line width=1.5pt,color=black,fill=black](73.center) to (74.center);
\draw [line width=1.5pt,color=black,fill=black](77.center) to (78.center);
\draw [line width=1.5pt,color=black,fill=black](78.center) to (79.center);
\draw [line width=1.5pt,color=gray!80,fill=gray!80](79.center) to (80.center);
\draw [line width=1.5pt,color=black,fill=black](67.center) to (60.center);
\draw [line width=1.5pt,color=black,fill=black](60.center) to (73.center);
\draw [line width=1.5pt,color=black,fill=black](81.center) to (87.center);
\draw [line width=1.5pt,color=black,fill=black](88.center) to (87.center);
\draw [line width=1.5pt,color=black,fill=black](89.center) to (90.center);
\draw [line width=1.5pt,color=black,fill=black](92.center) to (93.center);
\draw [line width=1.5pt,color=black,fill=black](93.center) to (94.center);
\draw [line width=1.5pt,color=gray!80,fill=gray!80](94.center) to (97.center);
\draw [line width=1.5pt,color=black,fill=black](95.center) to (96.center);
\draw [line width=1.5pt,color=black,fill=black](99.center) to (100.center);
\draw [line width=1.5pt,color=black,fill=black](100.center) to (101.center);
\draw [line width=1.5pt,color=gray!80,fill=gray!80](101.center) to (102.center);
\draw [line width=1.5pt,color=black,fill=black](81.center) to (95.center);
\draw [line width=1.5pt,color=black,fill=black](81.center) to (89.center);
\end{pgfonlayer}
\end{tikzpicture}}}
\caption{The trees  of Lemma \ref{beta2}.}\label{alphaTWO}
\end{figure}

\begin{lemma}\label{beta2}
Let $u\in  D_2$   and  $\beta(T_u)=2\alpha$.     Then,   $T_u$ is one  of the five trees indicated in Figure \ref{alphaTWO}, where all gray half-edges incident with $T_u$ are also shown.
\end{lemma}

\begin{proof}
As  $\beta(T_u)=2\alpha$, $T_u$ is not a single-vertex tree      and  so     $g(T_u)\geqslant1$   by    Lemma \ref{lCv}(i).
If  $g(u)\geqslant1$,  then we obtain   from       Lemma \ref{lCvSH}(ii),   \eqref{defwt},  and $\ell\geqslant14$    that
$$2\alpha=\beta(T_u)\geqslant \mathrm{wt}(u)-\epsilon(u)\alpha\geqslant \big(g(T_u)+(\ell-4)g(u)-\epsilon(u)\big)\alpha\geqslant(\ell-5)\alpha$$  a contradiction.
Hence,    $g(u)=0$.
If $g(T_u)=1$, then     Lemma \ref{lCv}(ii)  and  $\beta(T_u)=2\alpha$  yield    that    $T_u$ is a path  and the   unique  gray half-edge is incident with a $T_u$-pendant   vertex   $w\in D_{\ell-4}$. In this case,  $T_u$ has the  shape    illustrated  in Figure \ref{alphaTWO}(i).  So, in what follows, we let  $g(T_u)\geqslant2$.

First, assume that   $\mathrm{wt}(u)-\epsilon(u)\alpha=2\alpha$.
By applying  Lemma \ref{lCvSH}(ii)   for any vertex  $x\in N_{T_u}(u)$ with $\deg_G(x)\geqslant2$, we conclude that $\beta(T_x)=0$. This along with  Lemma \ref{beta0} implies that
$T_u-u$ is a vertex-disjoint   union of  some paths, each of which  has   a unique  gray half-edge  incident with  its endpoint     belonging to $D_{\ell-2}$.
Moreover, it follows from Lemma \ref{neigborone}(i)  that $\epsilon(u)=1$. By   $\mathrm{wt}(u)=2\alpha$,  $g(u)=0$,  and  \eqref{defwt},   we  get    that  $g(T_u)=3$.
Therefore, $\deg_{T_u}(u)=3$ and so,  in this case,  $T_u$ has the  shape   depicted  in Figure \ref{alphaTWO}(ii).

Next, assume that   $\mathrm{wt}(u)-\epsilon(u)\alpha=\alpha$.
By applying  Lemma \ref{lCvSH}(ii),  we deduce  that there exists  a vertex    $v\in N_{T_u}(u)$ such that  $\deg_G(v)\geqslant2$ and  $\beta(T_v)=\alpha$. In addition,
$\beta(T_x)=0$     for every  vertex  $x\in N_{T_u}(u)\setminus\{v\}$ with $\deg_G(x)\geqslant2$.  By  Lemma   \ref{beta1}, $T_v$ is one  of the three    trees indicated in Figure \ref{alpha}. In particular,   $g(T_v)\in\{1, 2\}$.
Working toward  a contradiction, suppose that  $\epsilon(u)=2$.  From     $g(u)=0$   and  \eqref{defwt}, we obtain that    $g(T_u)=3$.
It follows from  $g(T_v)\leqslant2$ and \eqref{grelation} that  there is a vertex    $v'\in N_{T_u}(u)\setminus\{v\}$ such that $g(T_{v'})\geqslant1$.
As   $\beta(T_{v'})=0$, Lemma   \ref{beta0}  yields that  $\deg_G(v')=2$,  contradicting    Lemma \ref{neigborone}(i).
This contradiction shows  that   $\epsilon(u)=1$.  It follows  from     $g(u)=0$   and  \eqref{defwt}  that    $g(T_u)=2$.
Let $T_v$ have  the  shape   depicted  in Figure \ref{alpha}(i).   Then,    $g(T_v)=1$ and so it follows from  $g(T_u)=2$ and \eqref{grelation} that  there is a vertex    $w\in N_{T_u}(u)\setminus\{v\}$ such that $g(T_{w})=1$. As   $\beta(T_{w})=0$, Lemma   \ref{beta0}  yields that   $T_{w}$ has the  shape   depicted  in Figure \ref{alpha0}.  Hence,   $T_u$ has the  shape   depicted  in Figure \ref{alphaTWO}(iii).
Let $T_v$ have  the  shape   depicted  in Figure \ref{alpha}(ii). Then,    $g(T_v)=2$ and so  $\deg_{T_u}(u)=2$ using   $\epsilon(u)=1$ and Lemma \ref{lCv}(ii). Thus,  in this case,  $T_u$ has the  shape   depicted  in Figure \ref{alphaTWO}(iv).
If $T_v$ has the  shape   depicted  in Figure \ref{alpha}(iii), then    $\epsilon(v)=2$ and  moreover,   $\deg_G(u)=2$ by     Lemma \ref{lCv}(i),  contradicting   Lemma \ref{neigborone}(i).

Finally,  assume that   $\mathrm{wt}(u)-\epsilon(u)\alpha=0$.  By     $g(T_u)\geqslant2$,  $g(u)=0$,  and  \eqref{defwt},   we  find      that  $g(T_u)=2$ and $\epsilon(u)=2$. Letting $v\in N_{T_u}(u)$ be    $G$-pendant, Lemma \ref{neigborone}(i)  implies that $\deg_G(x)\geqslant3$ for each vertex   $x\in N_{T_u}(u)\setminus\{v\}$.    This  along with  $g(T_u)=2$ and  Lemma \ref{lCv}(ii) implies  that $\deg_{T_u}(u)=2$.
Let $N_{T_u}(u)=\{v,w\}$ for some vertex $w$.  We have from \eqref{betaN} that  $\beta(T_{w})=2\alpha$. Also, from     $g(u)=0$, $g(T_u)=2$,   and   \eqref{grelation}, we deduce    that   $g(T_{w})=2$. Moreover,  it follows   from
$\deg_G(u)=3$ and  Lemma \ref{neigborone}(ii) that   $\epsilon(w)=1$.
Using    Lemma \ref{lCvSH}(ii) and  \eqref{defwt}, we derive   that    $$2\alpha=\beta(T_{w})\geqslant\mathrm{wt}(w)-\epsilon(w)\alpha\geqslant\alpha+ (\ell-6)g(w)\alpha.$$   As $\ell\geqslant14$,  we conclude  that    $g(w)=0$. Since  $\deg_G(w)\geqslant3$,   $\epsilon(w)=1$, and $g(T_{w})=2$, we   obtain from    Lemma \ref{lCv}(ii)  and  $\beta(T_{w})=2\alpha$       that   $T_{w}-w$ is a vertex-disjoint   union of  two  paths, one  of which  has   a unique  gray half-edge  incident with  its endpoint   belonging to $D_{\ell-2}$,  and the other one    has   a unique  gray half-edge  incident with  its endpoint   belonging to $D_{\ell-3}$.
Thus,    in this case,    $T_u$ has the  shape  indicated   in Figure \ref{alphaTWO}(v).
\end{proof}

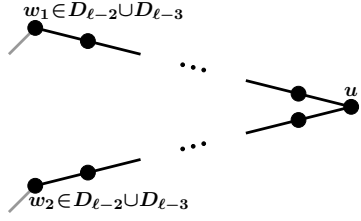
\begin{figure}
\centering
\scalebox{0.7}{\boldmath{
\begin{tikzpicture}
\begin{pgfonlayer}{nodelayer}
\draw [fill=black] (-3, 0) circle (4pt);
\draw [fill=black] (3, 1.5) circle (4pt);
\draw [fill=black] (2, 1.75) circle (4pt);
\draw [fill=black] (-2, 2.75) circle (4pt);
\draw [fill=black] (-3, 3) circle (4pt);
\draw [fill=black] (2, 1.25) circle (4pt);
\draw [fill=black] (-2, 0.25) circle (4pt);
\draw [fill=black] (0, 2.25) circle (1pt);
\draw [fill=black] (0.2, 2.2) circle (1pt);
\draw [fill=black] (-0.2, 2.3) circle (1pt);
\draw [fill=black] (0, 0.75) circle (1pt);
\draw [fill=black] (0.2, 0.8) circle (1pt);
\draw [fill=black] (-0.2, 0.7) circle (1pt);
\node [style=none] (0) at (3, 1.5) {$\bullet$};
\node [style=none] (1) at (3, 1.8) {$u$};
\node [style=none] (3) at (2, 1.75) {$\bullet$};
\node [style=none] (4) at (1, 2) {};
\node [style=none] (5) at (0, 2.25) {};
\node [style=none] (6) at (-1, 2.5) {};
\node [style=none] (7) at (-2, 2.75) {$\bullet$};
\node [style=none] (8) at (-3, 3) {$\bullet$};
\node [style=none] (9) at (2, 1.25) {$\bullet$};
\node [style=none] (10) at (1, 1) {};
\node [style=none] (11) at (-3.5, 2.5) {};
\node [style=none] (12) at (0, 0.75) {};
\node [style=none] (13) at (-1, 0.5) {};
\node [style=none] (14) at (-2, 0.25) {$\bullet$};
\node [style=none] (15) at (-3, 0) {$\bullet$};
\node [style=none] (16) at (-3.5, -0.5) {};
\node [style=none] (17) at (-1.7, 3.3) {$w_{1} \hspace{-1mm} \in \hspace{-1mm} D_{\ell -2}  \hspace{-1mm}  \cup  \hspace{-1mm} D_{\ell -3}$};
\node [style=none] (18) at (-1.6, -0.3) {$w_{2}\hspace{-1mm} \in \hspace{-1mm} D_{\ell -2}  \hspace{-1mm} \cup  \hspace{-1mm} D_{\ell -3}$};
\node [style=none] (19) at (3, 1.5) {};
\end{pgfonlayer}
\begin{pgfonlayer}{edgelayer}
\draw [line width=1.5pt,color=black,fill=black](0.center) to (3.center);
\draw [line width=1.5pt,color=black,fill=black](3.center) to (4.center);
\draw [line width=1.5pt,color=black,fill=black](6.center) to (7.center);
\draw [line width=1.5pt,color=black,fill=black](7.center) to (8.center);
\draw [line width=1.5pt,color=gray!80,fill=gray!80](8.center) to (11.center);
\draw [line width=1.5pt,color=black,fill=black](0.center) to (9.center);
\draw [line width=1.5pt,color=black,fill=black](9.center) to (10.center);
\draw [line width=1.5pt,color=black,fill=black](13.center) to (14.center);
\draw [line width=1.5pt,color=black,fill=black](14.center) to (15.center);
\draw [line width=1.5pt,color=gray!80,fill=gray!80] (15.center) to (16.center);
\end{pgfonlayer}
\end{tikzpicture}}}
\caption{The family $\mathcal{F}$.}\label{FFamily}
\end{figure}

Let   $u\in D_1$.  Reserve     $r(T_u)$ to denote the root adjacent to $u$.   We say    $T_u$ to be a  `positive-tree' if   $\beta(T_u)>0$, and   we  say    $T_u$ to be a  `null-tree' if  $\beta(T_u)=0$. The structure  of a null-tree  is   illustrated     in    Figure  \ref{alpha0}.
Set $$\mathcal{Z}_u=\big\{T_x \, \big| \,   \mbox{$T_x$ is  a null-tree   connected to $T_u$ through a gray edge}\big\}.$$
Define  an auxiliary function    $\gamma$    as follows.
For every  positive-tree  $T_u$ with  $x=r(T_u)$, define
$$\gamma(T_u)=\left\{\begin{array}{ll}
2\alpha & \mbox{ if $\deg_G(x)=3$,  $g(x)=0$,    and    $\epsilon(x)=2$;}\\
\alpha & \mbox{ if $\deg_G(x)\geqslant4$,  $g(x)=0$,  and   $\epsilon(x)=2$;}\\
\alpha & \mbox{ if    $g(x)=0$ and   $\epsilon(x)=1$;} \\
\frac{\alpha}{3} & \mbox{ if $g(x)>0$}
\end{array}\right.
$$
and, for every  null-tree $T_v\in \mathcal{Z}_u$ with $y=r(T_v)$, define
$$\gamma(T_v)=\left\{\begin{array}{ll}
\alpha & \mbox{ if $g(y)=0$;}\\
\frac{\alpha}{3} & \mbox{ otherwise.}
\end{array}\right.
$$
Denote by  $\mathcal{G}$   the family   of all      positive-trees $T_u$   indicated in    Figure  \ref{alpha}(i) with $\mathcal{Z}_u\neq\varnothing$.
Also, denote by   $\mathcal{F}$   the family  of all      positive-trees  $T_u$ indicated in    Figure  \ref{FFamily}  with      $g(r(T_u))=0$    and     $\epsilon(r(T_u))=2$   in which
a   $T_u$-pendant vertex           is    joined   to    a null-tree  if and only if it        belongs    to $D_{\ell-3}$.
Finally, denote by   $\mathcal{N}$     the family   of all     null-trees $T_u$   satisfying $\mathcal{Z}_u\neq\varnothing$.

In the following lemma, we   approximate the function  $\beta$  from below.

\begin{lemma}\label{gammalemma}
Let $T_u$ be a positive-tree  such that   $T_u\not\in \mathcal{G}\cup\mathcal{F}$. Then,
\begin{equation}\label{Ggamma}\beta(T_u)\geqslant \gamma(T_u)+\sum_{T_o\in  \mathcal{Z}_u}\gamma(T_o)\mbox{.}\end{equation}
\end{lemma}

\begin{proof}
For simplicity, set    $m=|\mathcal{Z}_u|$. Trivially,   $m\leqslant g(T_u)$.  Also,     $g(T_u)\geqslant1$  by    Lemma \ref{lCv}(i) and $\gamma(T_u)\leqslant(3-\epsilon(u))\alpha$  by   Lemma \ref{neigborone}(ii).
If  $g(u)>0$, then we find   from  \eqref{defwt} and  Lemma \ref{lCvSH}(ii) that
$$\beta(T_u)\geqslant \big(g(T_u)+(\ell-3)g(u)-\epsilon(u)\big)\alpha>\big(g(T_u)+2\big)\alpha\geqslant \gamma(T_u)+\sum_{T_o\in  \mathcal{Z}_u}\gamma(T_o),$$  as $\epsilon(u)\leqslant2$ and $\ell\geqslant14$.
So, in what follows, we  consider the constraints      $g(u)=0$ and   \begin{equation}\label{Ngamma}\beta(T_u)< \gamma(T_u)+\sum_{T_o\in  \mathcal{Z}_u}\gamma(T_o)\end{equation} and demonstrate that     $T_u\in \mathcal{G}\cup\mathcal{F}$.

If there is a vertex    $v\in V(T_u)\setminus\{u\}$ with  $\beta(T_v)\geqslant3\alpha$, then  we  obtain     from    \eqref{defwt},   Lemma \ref{lCvSH}(ii),   \eqref{Ngamma}, and $\gamma(T_u)\leqslant(3-\epsilon(u))\alpha$      that
$$\big(g(T_u)-\epsilon(u)+3\big)\alpha\leqslant\big(\mathrm{wt}(u)-\epsilon(u)\big)\alpha+\beta(T_v)\leqslant \beta(T_u)< \gamma(T_u)+\sum_{T_o\in  \mathcal{Z}_u}\gamma(T_o)\leqslant\big(m-\epsilon(u)+3\big)\alpha$$
which  means that  $g(T_u)<m$, a contradiction. This contradiction demonstrates  that    $\beta(T_v)\leqslant2\alpha$ for every vertex  $v\in V(T_u)\setminus\{u\}$.

Working toward a contradiction, suppose that there is a vertex  $v\in N_{T_u}(u)$  with   $g(T_v)\geqslant2$. As  $\beta(T_v)\leqslant2\alpha$,    Lemmas \ref{lCv}(i), \ref{beta0}, \ref{beta1}, and \ref{beta2}  yield that
$\beta(T_v)\in\{\alpha, 2\alpha\}$ and $T_v$ has a vertex $w\in  D_2\cup D_3$ such  that $T_w$ is
a starlike tree with center $w$ and  with at least two branches. Further,  any      endpoint   of   $T_w$  is   incident with  a  unique  gray edge.
Assume that  there are $p$   endpoints    of   $T_w$    each of which   belongs   to $D_{\ell-2}$ and
is joined  to a null-tree    and there are  $q$   endpoints    of   $T_w$        none of which     is    joined  to a null-tree.  Note that  $p+q=2$  if $\beta(T_v)=\alpha$.

We claim that  $p\neq0$. Suppose that this is not the case.  So,    $m\leqslant g(T_u)+\beta(T_v)/\alpha-3$.
We  get   from  \eqref{defwt},   Lemma \ref{lCvSH}(ii),  \eqref{Ngamma},   and $\gamma(T_u)\leqslant(3-\epsilon(u))\alpha$ that
$$\big(g(T_u)-\epsilon(u)\big)\alpha+\beta(T_v)\leqslant \beta(T_u)< \gamma(T_u)+\sum_{T_o\in  \mathcal{Z}_u}\gamma(T_o)\leqslant(m-\epsilon(u)+3)\alpha$$
which gives   $m> g(T_u)+\beta(T_v)/\alpha-3$,
a contradiction. This contradiction
establishes  the claim.

Let $ww_0w_1\ldots w_i$ be a  branch   of  $T_w$ with $w_i\in D_{\ell-2}$ so   that $w_i$  is adjacent to a vertex  $a_{\ell-3}\in D_{\ell-2}$ by a  gray  edge, where $i\in \{\ell-6, \ell-5\}$ and  the path $a_0a_1\ldots a_{\ell-3}$  belongs to  $\mathcal{Z}_u$ with $a_0\in D_1$. Also,   let $ww'_0w'_1\ldots w'_{j}$ be another  branch   of  $T_w$ with  $j\in\{\ell-7, \ell-6, \ell-5\}$. Now, $w_0w_1\ldots w_{i}a_{\ell-3}\ldots a_1a_0$ and $w'_0w'_1\ldots w'_{j}$ are two   paths     on degree-two vertices of   $G$     the  sum of whose   lengths   is $i+j+\ell-2\geqslant 2\ell-1$, as $\ell\geqslant14$.   So,  Lemma \ref{pkgeneral}  guarantees  the existence of   a  gray  edge joining $r(T_u)$ and $r(T_{a_0})$.
This implies  that  $\gamma(T_u)=\gamma(T_{a_0})=\alpha/3$.
We conclude  from  \eqref{defwt},   Lemma \ref{lCvSH}(ii),  and  \eqref{Ngamma}   that
$$\big(g(T_u)-2\big)\alpha+\beta(T_v)
\leqslant \beta(T_u)
<\gamma(T_u)+\sum_{T_o\in  \mathcal{Z}_u}\gamma(T_o)
\leqslant\left(\frac{1+p}{3}+\big(g(T_u)-p-q\big)\right)\alpha
$$
which gives $\beta(T_v)<\frac{7-2p-3q}{3}\alpha$,  providing     the required contradiction.

Therefore, in view of Lemma \ref{lCv}(i),  we   achieved that    $g(T_v)=1$ for every    vertex  $v\in N_{T_u}(u)$ with $\deg_G(v)\geqslant2$.
Since      $\beta(T_v)\leqslant2\alpha$ for every   vertex  $v\in N_{T_u}(u)$,  one  concludes     from   Lemma   \ref{lCv}(ii)      that
$T_u$ is a starlike tree with center $u$   each of  whose  endpoints      belongs      to $D_{\ell-2}\cup D_{\ell-3}\cup D_{\ell-4}$   and is   incident with  a unique   gray  edge.  In particular,  $g(T_u)=\deg_G(u)-1$.  Assume that  there are $r$ endpoints  of $T_u$    each of which   belongs   to $D_{\ell-2}$ and is joined  to a null-tree      and there are  $s$  endpoints of $T_u$ belonging    to $D_{\ell-2}$         none of which    is joined  to a null-tree. For convenience, set $t=\deg_G(u)-1-r-s$.

Let  $T_u$  be    a  path.
Lemma  \ref{neigborone}(i) implies that  $\epsilon(r(T_u))=1$ and so
$\gamma(T_u)\leqslant\alpha$.  By   \eqref{Ngamma},   we obtain   that   $\alpha\leqslant\beta(T_u)<\gamma(T_u)+\sum_{T_o\in  \mathcal{Z}_u}\gamma(T_o)\leqslant2\alpha$  which means that  $\beta(T_u)=\alpha$ and $m=1$.
Hence,    $T_u\in\mathcal{G}$.

Now, assume that  $T_u$  is not   a  path. This means that $\deg_G(u)\geqslant3$ and so $r+s+t\geqslant2$.
Toward  a contradiction,   suppose that  $r\neq0$.
Let $uu_0u_1\ldots u_{\ell-4}$ be  a     branch    of  $T_u$  with $u_{\ell-4}\in D_{\ell-2}$  such  that $u_{\ell-4}$ is adjacent to a vertex  $b_{\ell-3}\in D_{\ell-2}$ by a  gray  edge, where the path $b_0b_1\ldots b_{\ell-3}$ belongs to  $\mathcal{Z}_u$ with $b_0\in D_1$.   Further,  let $uu'_0u'_1\ldots u'_{h}$ be another  branch   of  $T_u$, where $h\in\{\ell-6, \ell-5, \ell-4\}$. Now, two   paths $u_0u_1\ldots u_{\ell-4}b_{\ell-3}\ldots b_1b_0$ and $u'_0u'_1\ldots u'_{h}$ are   on degree-two vertices of   $G$     the  sum of whose   lengths   is $(2\ell-5)+h> 2\ell-1$, as $\ell\geqslant14$.   So,  Lemma \ref{pkgeneral}  guarantees  the existence of   a  gray  edge joining $r(T_u)$ and $r(T_{b_0})$.
This implies  that  $\gamma(T_u)=\gamma(T_{b_0})=\alpha/3$. In view of    $\epsilon(u)=1$, we  get  from \eqref{defwt}, \eqref{defbeta}, and \eqref{Ngamma} that $$\big(g(T_u)-1+t\big)\alpha\leqslant \beta(T_u)<\gamma(T_u)+\sum_{T_o\in  \mathcal{Z}_u}\gamma(T_o)\leqslant\left(\frac{1+r}{3}+\big(g(T_u)-r-s\big)\right)\alpha.$$
This  gives $2r+3s+3t<4$ which   contradicts $r+s+t\geqslant2$.
Next, suppose  that $r=0$. We  obtain   from  \eqref{defwt}, \eqref{defbeta}, and \eqref{Ngamma} that
\begin{equation}\label{nemi}\big(g(T_u)-1+t\big)\alpha\leqslant \beta(T_u)<\gamma(T_u)+\sum_{T_o\in  \mathcal{Z}_u}\gamma(T_o)\leqslant\gamma(T_u)+\big(g(T_u)-s\big)\alpha\end{equation}
which gives $s+t<\gamma(T_u)/\alpha+1$.  This implies  that the  equalities hold in  the   inequalities $s+t\geqslant2$ and  $\gamma(T_u)\leqslant2\alpha$. We have  $\deg_G(u)=3$ and  $g(T_u)=2$.  Moreover,  we derive  from \eqref{nemi} that
$$\beta(T_u)=\left\{\begin{array}{ll}
\alpha & \mbox{ if $s=2$;}\\
2\alpha & \mbox{ if $s=1$;}\\
3\alpha & \mbox{ if $s=0$.}
\end{array}\right.$$
This yields  that        none   of    $T_u$-pendant vertices   belongs    to $D_{\ell-4}$. Thus,  $T_u\in\mathcal{F}$.
\end{proof}

In what follows, we investigate   the  trees from the family     $\mathcal{G}\cup\mathcal{F}\cup\mathcal{N}$.

\begin{lemma}\label{3Mcaseii}
$|\mathcal{F}|=O(1)$.
\end{lemma}

\begin{proof}
If $r(T_{u})$ and
$r(T_{u'})$  have no  common  neighbors in $M$ for any   two distinct trees     $T_u, T_{u'}\in\mathcal{F}$, then   $|\mathcal{F}|\leqslant |M|$, we are done.
So, assume that  there are  two distinct  trees     $T_u, T_{u'}\in\mathcal{F}$ such  that  two  vertices  $x=r(T_{u})$ and
$x'=r(T_{u'})$ have a  common  neighbor  in $M$, say $v$. In view of   $\deg_G(x)=\deg_G(x')=3$ and  $\epsilon(x)=\epsilon(x')=2$,  we  assume that   $N_G(x)=\{u, v, y\}$ and  $N_G(x')=\{u', v, y'\}$ for some   $G$-pendant vertices $y$ and $y'$.
As $G$ is   $C_{2\ell}$-saturated and $y$ is not adjacent to $y'$, there should be  a path $P$   of length $2\ell-1$ between  $y$ and $y'$.
If $v\in V(P)$  and  $u, u'\not\in V(P)$, then   the     length of $P$ would be   $4$,   contradicting  $\ell\geqslant14$.
Suppose for a contradiction and without loss of generality  that  $u, v\in V(P)$  and  $u'\not\in V(P)$,   then $P=yxuu_0u_1\ldots u_{\ell-4}Qvx'y'$, where  $Q$ is a path between $u_{\ell-4}$ and $v$ and  either   $uu_0u_1\ldots u_{\ell-4}$ is a branch of $T_u$ or   $uu_0u_1\ldots u_{\ell-5}$ is a branch of $T_u$ and $u_{\ell-5}u_{\ell-4}$ is a gray edge. As      $Q$ is of length  at least   $\ell-1$, the    length of $P$ is at least   $2\ell$, a contradiction. This contradiction shows  that  $u, u'\in V(P)$ and $v\not\in V(P)$.
If one of $T_u$ and $T_{u'}$ has no    endpoints  in $ D_{\ell-2}$, then     $P$ is of length  at least   $4\ell-4$ which is again a  contradiction.
Thus,   there are two    vertices $w\in V(T_u)\cap D_{\ell-2}$ and   $w'\in V(T_{u'})\cap D_{\ell-2}$  such that   $ww'$ is a gray edge.

We have shown  that,  for a  given  tree      $T_u \in\mathcal{F}$, there are  at most two    trees       $T_{u'} \in\mathcal{F}$ for which $N_G(r(T_{u}))\cap N_G(r(T_{u'}))\neq\varnothing$. This yields  that,   for a  given vertex      $v\in M$,   the number of  trees    from $\mathcal{F}$ connected to $v$   is at most $3$. Therefore,   $|\mathcal{F}|\leqslant 3|M|$ and  hence  $|\mathcal{F}|=O(1)$ by Lemma \ref{M2k}.
\end{proof}

If $\mathcal{Z}_u=\{T_v\}$  for  an  element   $T_u\in\mathcal{G}\cup\mathcal{N}$, then we say    $T_u$ and  $T_v$  are `mate'  of each other.
Recall that two mated trees  have different roots  by Lemma \ref{closedfener}.
We are going to   update  the values of  the  function $\gamma$        so that \eqref{Ggamma} holds for  every  positive-tree     $T_u\in  \mathcal{G}$ and its mate.  To do this, let $T_u$ be an arbitrary element of $\mathcal{G}$ and let $T_v$ be the mate  of $T_u$.   Since there is no path of length $2\ell-4$ on degree-two vertices in $V(G)\setminus M$,  either $\deg_G(r(T_u))\geqslant3$ or  $\deg_G(r(T_v))\geqslant3$. Now, we are making the following updates.
\begin{itemize}[noitemsep,  topsep=0pt]
\item[$\bullet$]   Update   $\gamma(T_u)$ and   $\gamma(T_v)$ to $\alpha$ and  $0$,  respectively,     if $\deg_G(r(T_u))=2$ and $\deg_G(r(T_v))\geqslant3$.
\item[$\bullet$]   Update   $\gamma(T_u)$ and   $\gamma(T_v)$ to $0$ and  $\alpha$,  respectively,    if $\deg_G(r(T_u))\geqslant3$ and $\deg_G(r(T_v))=2$.
\item[$\bullet$]   Update   $\gamma(T_u)$ and   $\gamma(T_v)$ to $\alpha/2$ and  $\alpha/2$,  respectively,   if $\deg_G(r(T_u))\geqslant3$ and $\deg_G(r(T_u))\geqslant3$.
\end{itemize}
Hereafter,  we disregard the terms   positive-tree and null-tree.
For each  vertex   $u\in D_1$, we say        $T_u$ to be a  `$\gamma$-tree' if $T_u\not\in\mathcal{N}$ and    $\gamma(T_u)=\gamma$. Also, we consider  each  tree from  $\mathcal{N}$ to be a $0$-tree.

Before we give a proof for Theorem  \ref{mainEC},  we    establish  two technical lemmas that will be used in the proof of the theorem.

\begin{lemma}\label{akh1}
Let   two  $0$-trees  have a shared    root      $x$. Then, $g(x)\neq0$.
\end{lemma}

\begin{proof}
We have $x\in B\cup D_0$. According to  the  definition of  $0$-tree, there are two paths  $xu_0u_1\ldots u_{2\ell-5}y$ and $xv_0v_1\ldots v_{2\ell-5}z$, where  $y, z\in M\cup B\cup D_0$ and   $\deg_G(u_i)=\deg_G(v_j)=2$ for all indices $i$ and $j$.
By applying    Lemma \ref{pkgeneral} for two   paths   $u_0u_1\ldots u_{2\ell-5}$ and $v_0v_1\ldots v_{2\ell-5}$  whose    lengths are     $2\ell-5$,  one finds    a   path of length $2$ between $x$    and $u_{2\ell-5}$ as well as there is   a   path of length $2$ between $x$    and $v_{2\ell-5}$. From  $\deg_G(u_{2\ell-5})=\deg_G(v_{2\ell-5})=2$, we deduce that $xy, xz\in E(G)$.
Thus, if either  $x\in B$ or $y\neq z$, then $x$  to be   joined to one of  $y$ and  $z$   through a  gray edge, as
required.

So,  assume   that $x\in D_0$,     $y=z$, and the edge $xy$ is black. We remark  that this happens if     $y\in M$. As $G$ is   $C_{2\ell}$-saturated and  $u_0v_0\not\in E(G)$, there is a path  of length $2\ell-1$  between $u_0$ and $v_0$.  Two paths   $u_0xv_0$ and  $u_0u_1\ldots u_{2\ell-5}yv_{2\ell-5}\ldots v_1v_0$ are of lengths $2$ and $4\ell-8$, respectively. All    remaining paths  between  between $u_0$ and $v_0$ have   the form    $P=u_0u_1\ldots u_{2\ell-5}yQxv_0$ for some path $Q$ between $x$ and $y$. Such a path $P$  is of length $2\ell-1$ if $Q$ has length $2$. This means   that  $Q=ywx$ for  some  vertex  $w\in M\cup B\cup D_0$. The  edge $xw$  is gray, as desired.
\end{proof}

\begin{lemma}\label{akh2}
There are no    a   $0$-tree     and   a   $\alpha/2$-tree       sharing   a  root      $x$ with   $\deg_G(x)=3$.
\end{lemma}

\begin{proof}
We have $x\in   D_0$. By the definition of  $0$-tree     and       $\alpha/2$-tree, there are two paths  $xu_0u_1\ldots u_{2\ell-5}y$ and $xv_0v_1\ldots v_{2\ell-6}z$, where  $y \in M\cup B\cup D_0$, $z  \in   B\cup D_0$,  and   $\deg_G(u_i)=\deg_G(v_j)=2$ for all indices $i$ and $j$.
By applying    Lemma \ref{pkgeneral} for two   paths   $u_0u_1\ldots u_{2\ell-5}$ and $v_0v_1\ldots v_{2\ell-6}$   whose    lengths   are   $2\ell-5$ and  $2\ell-6$,  respectively,  we find  a   path of length $2$ between $x$    and $u_{2\ell-5}$.  We conclude  from  $\deg_G(x)=3$ and  $\deg_G(u_{2\ell-5})=2$    that $x$ is adjacent to $y$ by a gray edge and   $y\in M$. In particular,   $y \neq z$.

Since   $G$ is   $C_{2\ell}$-saturated and  $u_0u_{2\ell-6}\not\in E(G)$, there exists   a path  of length $2\ell-1$  between $u_0$ and $u_{2\ell-6}$.   Two paths   $u_0xyu_{2\ell-5}u_{2\ell-6}$ and        $u_0u_1\ldots u_{2\ell-6}$ are of lengths $4$ and $2\ell-6$, respectively.   From  $\deg_G(x)=3$, one deduces that   all    remaining paths  between  between  $u_0$ and $u_{2\ell-6}$ have     the form    $P=u_0xv_0v_1\ldots v_{2\ell-6}zQyu_{2\ell-5}u_{2\ell-6}$ for some path $Q$ between $y$ and $z$. Such a path $P$  is of   length at least $2\ell$,  since  $Q$ has length at least $1$. This gives us  a contradiction.
\end{proof}

We are now in a position to prove our  main result of this section.
In the proof,   $G$ is as described above and  the notation  and definitions are used as introduced throughout this subsection.
Recall Theorem \ref{mainEC}.

\begin{reptheorem}{mainEC}
For each even     integer   $k\geqslant28$,
$$\mathrm{sat}(n, C_k)= \frac{k-3}{k-4}n +O(1).$$
\end{reptheorem}

\begin{proof}
As we  already mentioned,  using    the upper bound presented in \eqref{fur},   we should  establish  that $|E(G)|\geqslant (1+\alpha)|V(G)|+O(1)$.

Suppose that   $G$ has  many  paths of  length $2\ell-5$   on degree-two vertices  and there are two vertices $x, y\in V(G)$  such that one endpoint of
each of those paths is  adjacent to $x$ and the other endpoint is    adjacent to $y$.
In this case,  by removing all these paths from $G$ except two of them,  we find a $C_{2\ell}$-saturated graph $G'$  such that $|E(G')|-(1+\alpha)|V(G')|=|E(G)|-(1+\alpha)|V(G)|$.
Therefore, if $|E(G')|\geqslant (1+\alpha)|V(G')|+O(1)$, then we will have  $|E(G)|\geqslant (1+\alpha)|V(G)|+O(1)$.
So,  we may assume that   the number of such  paths in $G$  is at most two.

In order to establish that  $|E(G)|\geqslant (1+\alpha)|V(G)|+O(1)$, as we explained  before,  our strategy is to distribute the value $|E(G)|$ on  $V(G)\setminus (M\cup A)=B\cup D_0\cup D$  such  that each  vertex in $B\cup D_0\cup D$ takes an edge-weight at least   $1+\alpha$.
For each  vertex    $v\in D_0\cup D$, an edge-weight equal to $1$ comes from the  designated  black edge incident with $v$.
In view of  Lemma \ref{lCvSH}(ii), we assigned   an extra   edge-weight equal to  $\alpha$  to each     vertex in $D\setminus\{v\in D_1 \, | \, \deg_G(v)=1\}$.
As   every  $G$-pendant
vertex from $D_1$ is    a neighbor of some    vertex from  $D_0$, we need   to assign  an extra edge-weight  equal to  $\alpha$  to every vertex  $x\in D_0$ if
$\epsilon(x)=1$,  and $2\alpha$    otherwise.
Also, we need to assign an edge-weight equal to  $1+\alpha$   to    every  vertex   $y\in B$  if
$\epsilon(y)=1$,   and $1+2\alpha$  otherwise. In view of  Lemma \ref{3Mcaseii},  it suffices to do these assignments to those  vertices   $x$ and $y$ that are roots of  none of  elements of  $\mathcal{F}$.

Let $x\in D_0$. By the definition  of $D_0$, $\deg_G(x)\geqslant2$.      If   $g(x)\geqslant1$, then  we may assign   an  extra   edge-weight   $1/2>2\alpha$ to $x$, we are done.
So,   we may assume that   $g(x)=0$. This implies that there is no  $\alpha/3$-tree with root $x$.  Assume that      $\gamma_1,  \ldots, \gamma_d$  are all values, not necessarily distinct,  such that $x$ is  adjacent to a   $\gamma_i$-tree    for $i=1,  \ldots, d$ and assume that  $d\geqslant2$. From   Lemma  \ref{akh1},  there is at most one $0$ among  $\gamma_1,  \ldots, \gamma_d$.   Moreover,   if $\deg_G(x)=3$, then  Lemma  \ref{akh2} implies that   $0$ and  $\alpha/2$ do not simultaneously appear among   $\gamma_1,  \ldots, \gamma_d$.     From these, we    deduce  that       $\gamma_1+\cdots+\gamma_d\geqslant\alpha$ and so  we may assign  an  extra    edge-weight  equal to  $\alpha$ to $x$.  Hence, we may assume that either $\deg_G(x)=2$ or $\epsilon(x)=2$.

First, assume that $\deg_G(x)=2$. We have  from   Lemma \ref{neigborone}(i)    that $\epsilon(x)=1$ and  so   there exists  a unique  $\gamma$-tree with root $x$. In view of      $\deg_G(x)=2$, we have $\gamma\neq\alpha/2$. For a contradiction, let   $\gamma=0$. This yields that
there is a path    on degree-two vertices of $G$ of length $2\ell-5$ whose one  endpoint
is  adjacent to $x$ and whose   other endpoint belongs to  $B\cup D_0$. From this and $\deg_G(x)=2$, we  get  a path    on degree-two vertices   of length $2\ell-4$ in $V(G)\setminus M$  which contradicts the definition of $M$.   This contradiction shows that  $\gamma=\alpha$. Now,   by  assigning    an      extra edge-weight  equal to $\alpha$       to $x$,  we are done.

Next, assume that $\epsilon(x)=2$. If  $\deg_G(x)=3$, then  there exists  a    $2\alpha$-tree with root $x$ and so  we may assign  an  extra    edge-weight  equal to  $2\alpha$ to $x$, we are done. If    $\deg_G(x)\geqslant4$, then   there are two     $\alpha$-trees   whose roots are $x$    and so we may assign  an  extra    edge-weight  equal to  $2\alpha$ to $x$, we are done.

Let $y\in B$. We are going to    assign an edge-weight equal to  $1+2\alpha$   to   $y$.
Denote by $h$    the number of       $0$-trees  with   root   $y$.  According to the definition of $0$-tree and by applying the above-mentioned property of $G$, we find  mutually    distinct  vertices  $y_1,   \ldots, y_{\lceil h/2\rceil} \in M\cup B\cup D_0$    and mutually vertex-disjoint paths  $P_1,   \ldots, P_{\lceil h/2\rceil}$  on degree-two vertices   of  $G$    having  length   $2\ell-5$ such that one endpoint of
$P_i$ is  adjacent to $y$ and its  other endpoint is    adjacent to $y_i$  for all indices  $i$.
It follows from
Lemma \ref{pkgeneral}  that $y$ is adjacent to all of $y_1,   \ldots, y_{\lceil h/2\rceil}$  through gray edges, meaning that    $g(y)\geqslant h/2$.
So, we have an extra edge-weight   at least    $g(y)(1/2)+(\deg_G(y)-g(y)-h-1)(\alpha/3)>(c-1)\alpha/3=1+2\alpha$, as     $g(y)\geqslant h/2$ and $\ell\geqslant14$. Now, by assigning    this    extra edge-weights      to $y$,  we are done.

The proof of the theorem is completed.
\end{proof}

\section{Concluding remarks}

In this paper, we establish that   there  is    a constant $c_{_{\mathlarger{F}}}$  such
that  $\mathrm{sat}(n, F)=c_{_{\mathlarger{F}}} n+O(1)$ when  $F$ is either a complete  multipartite graph  or  a    cycle graph   whose   length is an    even  number at least $28$. We believe that  our method can be applied  to find an asymptotic  formula for   $\mathrm{sat}(n, F)$ when  $F$ is a   cycle graph   whose   length is  a    small  even  or an odd number.
Using  the  method,  we were   able to reprove      the known result       $\mathrm{sat}(n, C_6)=4n/3  +O(1)$ with a  rather short and simple  proof.

\end{document}